\documentclass[a4paper, 11pt, english]{article}

\usepackage[geometry]{ifsym}

\usepackage{pgf}
\usepackage{amsmath,amssymb}
\usepackage{mathrsfs}
\usepackage{bbm}
\usepackage{verbatim}
\usepackage{epsfig}
\usepackage{graphicx}
\usepackage[all]{xy}
\usepackage{color}
\usepackage{import}
\usepackage{bussproofs}
\usepackage{qtree}
\usepackage{tree-dvips}
\usepackage{lingmacros}
\usepackage{amsfonts}
\usepackage{mathdots}
\usepackage{amssymb}
\usepackage{pifont}
\usepackage[safe]{tipa}
\usepackage{newlfont}
\usepackage{latexsym}
\usepackage{multirow}
\usepackage{array,cmll}
\usepackage{textcomp}
\usepackage{accents}
\usepackage{rotating}
\usepackage{logreq}

\usepackage{hyperref}

  \usepackage{amsmath}
  \usepackage{amsthm}
  \usepackage{latexsym}
  \usepackage{amscd}
  \usepackage{amssymb}
  \usepackage{bussproofs}
  \usepackage{txfonts}

  \usepackage{tikz}
  \usepackage[position=top]{subfig}

  \usepackage{leqno}

\usepackage{extpfeil}

\input theorem.sty

\def\fCenter{{\mbox{$\ \vdash\ $}}}

\EnableBpAbbreviations

\newcommand{\commment}[1]{}

\def\aol{\rule[0.5865ex]{1.38ex}{0.1ex}}

\def\pdra{\mbox{$\,>\mkern-8mu\raisebox{-0.065ex}{\aol}\,$}}
\def\pdla{\mbox{\rotatebox[origin=c]{180}{$\,>\mkern-8mu\raisebox{-0.065ex}{\aol}\,$}}}

\def\gSCRA{\mbox{$\succ$}}
\def\gSCLA{\mbox{$\prec$}}

\def\gscra{\mbox{$\,\raisebox{-0.065ex}{\aol}\mkern-5.5mu\raisebox{0.155ex}{${\scriptstyle{\succ}}$}\,$}}
\def\gscla{\mbox{\rotatebox[origin=c]{180}{$\,\raisebox{-0.065ex}{\aol}\mkern-5.5mu\raisebox{0.155ex}{${\scriptstyle{\succ}}$}\,$}}}

\def\mANDORatom#1{\hbox{\hbox to 0pt{$#1\TriangleUp$\hss}$#1\TriangleDown$}}

\newcommand{\mcAND}{%
\mathrel{\ooalign{\raisebox{-0.39ex}{$\mbox{\TriangleUp}$}\cr\kern4.2pt{\raisebox{-0.13ex}{$\cdot$}}}}}
\newcommand{\mcand}{%
\mathrel{\ooalign{$\vartriangle$\cr\kern1.99pt{\raisebox{-0.17ex}{$\cdot$}}}}}

\newcommand{\mAND}{\raisebox{-0.39ex}{\mbox{\,\TriangleUp\,}}}
\newcommand{\mand}{\vartriangle}

\newcommand{\nAND}{%
\mathrel{\ooalign{$\mbox{\TriangleUp}$\cr\kern0pt$\mbox{\rotatebox[origin=c]{180}{\TriangleUp}}$}}}

\newcommand{\nand}{%
\mathrel{\ooalign{$\vartriangle$\cr\kern0pt$\triangledown$}}}

\newcommand{\mcBAND}{%
\mathrel{\ooalign{\raisebox{-0.39ex}{$\mbox{\FilledTriangleUp}$}\cr\kern4.2pt{\raisebox{-0.13ex}{${\color{white}\cdot}$}}}}}
\def\mBAND{\raisebox{-0.39ex}{\mbox{\,\FilledTriangleUp\,}}}
\def\mband{\mbox{$\mkern+2mu\blacktriangle\mkern+2mu$}}

\newcommand{\mcband}{%
\mathrel{\ooalign{$\blacktriangle$\cr\kern1.99pt{\raisebox{-0.17ex}{${\color{white}\cdot}$}}}}}

\newcommand{\mcRA}{%
\mathrel{\ooalign{
                  \raisebox{-0.3ex}{$\rotatebox[origin=c]{-90}{$\mbox{{\TriangleUp}}$}$}
                                                                            \cr\kern2.7pt{\raisebox{0.2ex}{$\cdot\mkern1.3mu$}}}}}

\def\mRA{\mbox{\,\raisebox{-0.39ex}{\rotatebox[origin=c]{-90}{\TriangleUp}}\,}}
\newcommand{\mra}{\mbox{$\,-{\mkern-3mu\vartriangleright}\,$}}
\newcommand{\mbra}{\mbox{$\,-{\mkern-3mu\blacktriangleright}\,$}}

\newcommand{\mcra}{%
\mathrel{\ooalign{$\,{\vartriangleright\,}$\cr\kern3pt{\raisebox{0ex}{$\cdot$}}}}}

\newcommand{\mcraline}{%
-{\mkern-6mu{\mathrel{\ooalign{$\,{\vartriangleright\,}$\cr\kern3pt{\raisebox{0ex}{$\cdot$}}}}}}}

\newcommand{\mdraline}{%
{\mathrel{\ooalign{$\,{\vartriangleright\,}$\cr\kern3pt{\raisebox{0ex}{$\cdot$}}}}}{\mkern-6mu}-}

\newcommand{\cra}{%
\mathrel{\ooalign{$\,-{\mkern-3mu\vartriangleright\,}$\cr\kern8pt{\raisebox{0ex}{$\cdot$}}}}}

\def\mBRA{\mbox{\,\raisebox{-0.39ex}{\rotatebox[origin=c]{-90}{\FilledTriangleUp}}\,}}
\newcommand{\mcBRA}{%
\mathrel{\ooalign{
                  \raisebox{-0.3ex}{$\rotatebox[origin=c]{-90}{$\mbox{\FilledTriangleUp}$}$}
                                                                            \cr\kern2.7pt{\raisebox{0.2ex}{${\color{white}\cdot}$}}}}}

\newcommand{\mcbra}{%
\mathrel{\ooalign{$\,-{\mkern-3mu\blacktriangleright\,}$\cr\kern8pt{\raisebox{0ex}{$\cdot$}}}}}

\def\mLA{\mbox{\,\raisebox{-0.39ex}{\rotatebox[origin=c]{90}{\TriangleUp}}\,}}
\newcommand{\mcLA}{%
\mathrel{\ooalign{
                  \raisebox{-0.3ex}{$\rotatebox[origin=c]{90}{$\mbox{\TriangleUp}$}$}
                                                                                     \cr\kern5.5pt{\raisebox{0.2ex}{$\cdot$}}
                                                                                                                              }}}

\def\nla{\mbox{$\,\vartriangleleft\,$}}

\newcommand{\mcla}{%
\mathrel{\ooalign{$\,{\vartriangleleft\,}$\cr\kern5pt{\raisebox{0ex}{$\cdot$}}}}}

\newcommand{\mclaline}{%
-{\mkern-6mu{\mathrel{\ooalign{$\,{\vartriangleleft\,}$\cr\kern5pt{\raisebox{0ex}{$\cdot$}}}}}}}

\def\nla{\mbox{$\,\vartriangleleft\,$}}

\def\mSLA{\mbox{\,\raisebox{-0.4ex}{\rotatebox[origin=c]{90}{\TriangleUp}}\raisebox{0.20ex}{$\mkern-2mu\thicksim$}\,}}

\def\msla{\mbox{$\,\vartriangleleft{\mkern-8mu\sim}\,$}}

\def\mBLA{\mbox{\,\raisebox{-0.39ex}{\rotatebox[origin=c]{90}{\FilledTriangleUp}}\,}}
\newcommand{\mcBLA}{%
\mathrel{\ooalign{
                  \raisebox{-0.3ex}{$\rotatebox[origin=c]{90}{$\mbox{\FilledTriangleUp}$}$}
                                                                                     \cr\kern5.5pt{\raisebox{0.2ex}{${\color{white}\cdot}$}}
                                                                                                                                            }}}

\def\nbla{\mbox{$\,\blacktriangleleft\,$}}

\def\mSBLA{\mbox{\,\raisebox{-0.4ex}{\rotatebox[origin=c]{90}{\FilledTriangleUp}}\raisebox{0.20ex}{$\mkern-2mu\thicksim$}\,}}

\def\msbla{\mbox{$\,\blacktriangleleft{\mkern-8mu\sim}\,$}}

  \numberwithin{equation}{section}

\newcommand{\ls}{\lbrack}
\newcommand{\rs}{\rbrack}
\newcommand{\lc}{\langle}
\newcommand{\rc}{\rangle}

\newcommand{\pand}{\wedge}

\newcommand{\por}{\vee}

\newcommand{\pra}{\rightarrow}
\newcommand{\plra}{\leftrightarrow}
\newcommand{\pla}{\leftarrow}
\newcommand{\gI}{%
\mathrel{\ooalign{$\mbox{T}$\cr\kern0pt$\mbox{\rotatebox[origin=c]{180}{T}}$}}}
\newcommand{\gtop}{\tau}
\newcommand{\gbot}{\rotatebox[origin=c]{180}{$\tau$}}
\newcommand{\dwn}{\rotatebox[origin=c]{180}{$\wn$}}
\newcommand{\DWN}{\rotatebox[origin=c]{180}{{\boldmath{$\wn$}}}}
\newcommand{\WN}{\rotatebox[origin=c]{}{{\boldmath{$\wn$}}}}
\newcommand{\gC}{\between}
\newcommand{\gsc}{\mbox{$\,;\,$}}
\newcommand{\gSC}{\mbox{\boldmath{{\Large{$\,;\,$}}}}}
\newcommand{\gand}{\cap}
\newcommand{\gAND}{\mbox{$\,\bigcap\,$}}
\newcommand{\gor}{\cup}

\def\gra{\mbox{$\,{\raisebox{0.065ex}{\aol}{\mkern-1.4mu{\rotatebox[origin=c]{-90}{\raisebox{0.12ex}{$\gand$}}}}}\,$}}
\def\gRA{\mbox{$\,\,{{\rotatebox[origin=c]{-90}{\raisebox{0.12ex}{$\gAND$}}}}\,\,$}}

\def\gla{\mbox{\rotatebox[origin=c]{-180}{$\,{\raisebox{0.065ex}{\aol}{\mkern-1.4mu{\rotatebox[origin=c]{-90}{\raisebox{0.12ex}{$\gand$}}}}}\,$}}}
\def\gLA{\mbox{\rotatebox[origin=c]{180}{$\,\,{{\rotatebox[origin=c]{-90}{\raisebox{0.12ex}{$\gAND$}}}}\,\,$}}}

\def\gdra{\mbox{$\,{\rotatebox[origin=c]{-90}{\raisebox{0.12ex}{$\gand$}}{\mkern-1mu\raisebox{0.065ex}{\aol}}}\,$}}

\def\gdla{\mbox{\rotatebox[origin=c]{-180}{$\,{\rotatebox[origin=c]{-90}{\raisebox{0.12ex}{$\gand$}}{\mkern-1mu\raisebox{0.065ex}{\aol}}}\,$}}}

\def\aga{\texttt{a}}

\def\aol{\rule[0.5865ex]{1.38ex}{0.1ex}}

\newcommand{\RESalphaBox}{\,\rotatebox[origin=c]{90}{$[\rotatebox[origin=c]{-90}{$\alpha$}]$}\,}
\newcommand{\RESalphaplusBox}{\,\rotatebox[origin=c]{90}{$\Big[\rotatebox[origin=c]{-90}{$\,\,\alpha^+$}\Big]$}\,}

\newcommand{\RESalphaDia}{\,\rotatebox[origin=c]{90}{$\langle\rotatebox[origin=c]{-90}{$\alpha$}\rangle$}\,}
\newcommand{\RESalphaplusDia}{\,\rotatebox[origin=c]{90}{$\Big\langle\rotatebox[origin=c]{-90}{$\,\,\alpha^+$}\Big\rangle$}\,}

\makeatletter
\newcommand{\WKnowProxy}[2]{%
  {\mathbin{\ooalign{$#1\circ#2 $\cr\hidewidth
   \raise.155ex\hbox{$#1{\scriptstyle{\ast}}#2$}\hidewidth\cr  }}}}
\makeatother

\makeatletter
\newcommand{\BKnowProxy}[2]{%
  {\mathbin{\ooalign{$#1\bullet#2 $\cr\hidewidth
   \raise.155ex\hbox{$#1{\scriptstyle{\color{white}{\ast}}}#2$}\hidewidth\cr  }}}}
\makeatother









\newcommand{\fns}{\footnotesize}
\newcommand{\mc}{\multicolumn}

\title{Multi-type display calculus for Propositional Dynamic Logic}
\date{}

\author{Sabine Frittella\footnote{Laboratoire d'Informatique Fondamentale de Marseille (LIF) - Aix-Marseille Universit\'{e}.}\;\; Giuseppe Greco\footnote{Department of Values, Technology and Innovation - TU Delft.}\\ Alexander Kurz\footnote{Department of Computer Science - University of Leicester.}\;\; Alessandra Palmigiano\footnote{Department of Values, Technology and Innovation - TU Delft. The research of the second and fourth author has been made possible by the Vidi grant 016.138.314 of the Netherlands Organization for Scientific Research (NWO).}}

\newtheorem{proposition}[theorem]{Proposition}
\newtheorem{corollary}[theorem]{Corollary}

\begin{document}
\maketitle


\begin{abstract}
We introduce a multi-type display calculus for Propositional Dynamic Logic (PDL). This calculus is complete w.r.t.\ PDL, and enjoys Belnap-style cut-elimination and subformula property.
{\em Math.\ Subject Class.\ 2010:} 	03B42,  03B20, 03B60, 03B45, 03F03, 03G10, 03A99.
\end{abstract}

\tableofcontents

\section{Introduction}
\label{PDL:sec:Introduction}
   The present paper introduces a cut-free, multi-type display calculus for the full language of Propositional Dynamic Logic (PDL). Methodologically, the contribution of the present paper follows the `multi-type' approach introduced in \cite{Multitype}. This approach is inspired by ideas developed in the area of proof-theoretic semantics (cf.\ \cite{GAV} for more on this point), and aims at facilitating the build-up of a modular and uniform proof-theory of dynamic logics. In \cite{Multitype}, the dynamic epistemic logic (DEL) of Baltag-Moss-Solecki is accounted for, by embedding the original DEL-language into a more expressive language in which not only formulas are generated from formulas and actions (as it happens in the symbol $\langle\alpha \rangle A$) and formulas are generated from formulas and agents (as it happens in the symbol $\langle\aga \rangle A$), but also {\em actions} are generated from the interactions between {\em agents} and {\em actions}.

The key idea of this approach is that the new language is {\em multi-typed}: namely, actions, agents and formulas are all first-class citizens of the language, each of them belongs to at least one type, and each generation step mentioned above is accounted for in an explicit way via special connectives taking arguments of {\em different types}. Specifically, the multi-type DEL-setting of \cite{Multitype} consists of the following types: $\mathsf{Ag}$ for agents, $\mathsf{Fnc}$ for functional actions, $\mathsf{Act}$ for general actions, and $\mathsf{Fm}$ for formulas. An important technical solution adopted in the DEL-setting, which will be relevant also to the present treatment of PDL, is that (epistemic) actions are assigned {\em two} distinct types: indeed, $\mathsf{Fnc}$ is the basic type of epistemic actions, and essentially reflects the action-parameters of the dynamic modal connectives in the original DEL-language, whereas $\mathsf{Act}$ is the type of actions as they appear to a given agent, and they have no explicit counterpart in the original language.

The introduction of different types lays the ground for overcoming some of the difficulties which typically make the proof-theoretic treatment of dynamic logics not straightforward, and have hindered the smooth transfer of results from one dynamic logic to another. In the specific case of PDL, the main hurdle lies in the encoding of the induction axiom into a structural rule. Due to the inductive `loop' (by which we mean a given parametric formula occurring both in antecedent and in consequent position), the induction axiom cannot straightforwardly be captured at the structural level, given that structures mean different things depending on their (antecedent or consequent) position in a sequent. However, as it is well known,(cf.\ \cite{Kozen2000}), the  induction axiom reflects, in the context of {\em formulas}, the fact that the \emph{positive iteration} + (resp.\ \emph{iteration}, also known as \emph{Kleene star} $\ast$) semantically is the (reflexive and) transitive closure operator {\em on actions}. This fact can (and should) be captured by finitary (structural) rules {\em purely in the context of actions}, that is, without formulas playing a mediating role. As in the case of DEL (cf.\ \cite{Multitype}), the multi-type environment can be used to provide the additional expressivity needed to capture the informational contents of various dynamic axioms in the context most naturally suited to support them.

Summing up, the multi-type methodology aims at generating calculi enjoying the following features, as witnessed e.g.\ by the calculus defined in \cite{Multitype}: (1) a neat division of labour, required by proof-theoretic semantics, between operational and structural rules; (2) a neat division of labour  between structural rules describing the properties {\em pertinent to each type}, and structural rules describing the {\em interaction} between different types; (3) all rules enjoying closure under uniform substitution of parametric operational terms for arbitrary structures within each type. 

\noindent Feature (3) is  crucial to the definition of the {\em multi-type} version of {\em  properly displayable calculi} (cf.\ Wansing's definition \cite[Section 4.1]{Wa98}. See Section \ref{PDL:ssec:quasi-def} for more on this). The cut elimination result for any such calculus  follows straightforwardly from the corresponding Belnap-style metatheorem \cite[Theorem 3.3]{Multitype}.

The multi-type display calculus for PDL introduced in the present paper has a design similar to the one in \cite{Multitype}, from which features analogous to (1) and (3) above will follow. This calculus is shown to be properly displayable (cf.\ Definition \ref{PDL:def:quasi-displayable multi-type calculus}), and hence  its cut elimination and subformula property smoothly follow from the Belnap-style metatheorem \cite[Theorem 3.3]{Multitype}. In this calculus, both actions and formulas are first-class citizens; each of them belongs to at least one type; connectives similar to the ones introduced in \cite{Multitype} account for the generation of compound terms taking arguments of different types. Specifically, the multi-type PDL-setting consists of the following types: $\mathsf{TAct}$ for {\em transitive actions}, $\mathsf{Act}$ for general actions, and $\mathsf{Fm}$ for formulas. As it was the case of the setting of \cite{Multitype}, in the present setting actions are assigned {\em two} distinct types: indeed, $\mathsf{Act}$ is the basic type of actions, and encodes the generic action-parameters of the dynamic modal connectives in the original PDL-language, whereas the type $\mathsf{TAct}$ of {\em transitive actions} is home to the $\alpha^+$  action-terms in the original language. The interaction between the two types of actions makes it possible to capture transitive closure, at the operational level, purely in the context of actions.

However, the present system does not enjoy  feature (2) above. Indeed, the present setting is not yet expressive enough to capture transitive closure at the {\em structural} level. Hence, the induction axiom is captured by means of infinitary rules, and making use of the mediating role of formulas. We conjecture that being able to express transitive closure at the structural level is key to dispensing with the infinitary rules, which is our next goal for future developments in this line research.

The organization of the paper goes as follows:  in Section \ref{PDL:sec:Preliminaries}, we collect the relevant preliminaries on PDL, and we recall the  generalization of the notion of proper display calculi to the multi-type setting, and its corresponding extension of Belnap's cut elimination metatheorem. 
In Section \ref{PDL:sec:language and rules}, we introduce the multi-type display calculus for PDL, which we refer to as Dynamic Calculus for PDL. 
In Section \ref{PDL:sec:Soundness}, we discuss the soundness of some of the rules w.r.t.\ the standard semantics. 
In Section \ref{PDL:sec:Completeness},
we prove that this calculus is complete w.r.t.\ PDL.
In Section \ref{PDL:sec:Cut-elimination},
we prove that it enjoys the Belnap-style cut elimination. 
In Section \ref{PDL:sec:Conservativity}, we discuss the different techniques available to prove that the calculus is conservative.
In Section \ref{PDL:sec:Conclusions}, we collect some conclusions and indicate further directions. Most of the proofs and derivations are collected in  Appendices \ref{Appendix : Cut elimination for PDL, principal stage} and \ref{Appendix : completeness for PDL}.

\section{Preliminaries}
\label{PDL:sec:Preliminaries}
   In the present section, we collect some basic facts: in subsection \ref{PDL:ssec:PDL}, we briefly review the Hilbert-style axiomatization of Propositional Dynamic Logic (PDL);  in subsection \ref{PDL:ssec:quasi-def} we introduce the definition of proper  multi-type display calculi, and state their cut elimination metatheorem. This definition and theorem are more compact versions of the ones of \cite{Multitype}, hence they will not be expanded on, but the Belnap-style cut elimination for the Dynamic Calculus for PDL (cf.\ Section \ref{PDL:sec:Cut-elimination}) will be shown on the basis of this more compact definition.

\subsection{Propositional Dynamic Logic}
\label{PDL:ssec:PDL}
In our review of PDL, we will loosely follow \cite{Kozen2000}. However, our presentation differs from that in \cite{Kozen2000} in some respects, which will be discussed below.

Let \textsf{AtProp} and \textsf{AtAct} be  countable and disjoint sets of atomic propositions and atomic actions, respectively. The set $\mathcal{L}$ of the formulas $A$ of Propositional Dynamic Logic (PDL), and the set $\mathsf{Act}(\mathcal{L})$ of the {\em actions} $\alpha$ {\em over} $\mathcal{L}$ are defined simultaneously as follows:
\begin{center}
$A ::= p\in \mathsf{AtProp} \mid \neg A \mid A\vee A \mid \langle\alpha\rangle A\;\; (\alpha\in \mathsf{Act}(\mathcal{L})),$
\end{center}
\begin{center}
$\alpha ::= a\in \mathsf{AtAct} \mid \alpha \gsc \alpha \mid \alpha\cup\alpha  \mid A \wn  \mid \alpha^+\;\; (A\in \mathcal{L})$.
\end{center}
Let the symbols $\wedge$ and  $\bot$ be defined as usual, that is $A \wedge B := \neg (\neg A \vee \neg B)$ and $\bot := A \wedge \neg A$ for some formula $A \in \mathcal{L}$.\\
Models for this language are  tuples $M = (W, \{R_a\mid a\in \mathsf{AtAct}\}, V)$,  such that $R_a\subseteq W\times W$ for each $a\in \mathsf{AtAct}$ and $V(p) \subseteq W$ for each $p \in \mathsf{AtProp}$. For each non atomic $\alpha\in \mathsf{Act}(\mathcal{L})$, the relation $R_\alpha\subseteq W\times W$ is defined recursively as follows:
\begin{center}
\begin{tabular}{r c l l}
$R_{\alpha \gsc \beta}$      & := & $R_\alpha \circ R_\beta$  &  $= \{(u, v) \mid \exists w \in W ((u, w) \in R_\alpha \,\textrm{ and }\, (w, v) \in R_\beta)\}$ \\
$R_{\alpha \,\gor\, \beta}$      & := & $R_\alpha \gor R_\beta$  &  $= \{(u, v) \mid (u, v) \in R_\alpha \,\textrm{ or }\, (u, v) \in R_\beta\}$ \\
$R_{\alpha^+}$                  & := &  $(R_\alpha)^+$                & $= \bigcup_{n\geq 1}R_\alpha^{n}$ \\
$R_{A \wn }$                       & := & \mc{2}{l}{$\{(w, w)\mid  w\in V(A)\}$.}   \\

\end{tabular}
\end{center}
where $R_\alpha^n$ is defined as follows: $R_\alpha^1:= R$ and $R_\alpha^{n+1} := R_\alpha^n \circ R_\alpha$ for every $n\geq 1$.
{\commment{
\begin{center}
\begin{tabular}{r c l c r c l}
$R_{\alpha \,;\, \beta}$       & := & $R_\alpha \circ R_\beta$              & $\qquad$ & $R_{\alpha \,\cup\, \beta}$ & := & $R_\alpha \cup R_\beta$  \\
$R_{\alpha \,\cap\, \beta}$ & := & $R_\alpha \cap R_\beta$              & $\qquad$ & $R_{\wn A}$                       & := & $\{(w, w)\mid  w\in V(A)\}$ \\
$R_{\alpha^+}$                  & := & $\bigcup_{n\geq 1}R_\alpha^{n}.$ &                &                                            &     &                                            \\
\end{tabular}
\end{center}
}}
Given the stipulations above, the definitions of satisfiability and validity of propositions are the usual ones in
modal logic; in particular the evaluation of formulas of the form $\langle\alpha\rangle A$ (resp.\ $\ls\alpha\rs A$) makes use of the  corresponding relation $R_\alpha$:

\begin{center}
\begin{tabular}{l}
$M, w \Vdash \lc\alpha\rc A$\quad iff \quad $\exists v \in W (w R_\alpha v \,\textrm{ and }\, M, w \Vdash A)$                   \\
$M, w \Vdash \ls\alpha\rs A$\quad iff \quad $\forall v \in W (\textrm{if }\, w R_\alpha v \,\textrm{ then }\, M, w \Vdash A)$ \\
\end{tabular}
\end{center}
The following list of axioms and rules constitutes a sound and complete Hilbert-style deductive system for PDL with box as primitive modality and
diamond defined as usual $\lc\alpha\rc A := \neg \ls\alpha\rs \neg A$. We list the \emph{Fix-Point Axiom} (also called \emph{Mix Axiom}) and the \emph{Induction Axiom} (also called \emph{Segerberg Axiom}) in the version with $\ast$ (cf.\ \cite{Seg,FisLad,Kozen2000}) and with $+$ (cf.\ \cite{Har,Har13}). For proofs of the completeness of PDL, the reader is referred to \cite{KP81}.
For the sake of the developments of the following sections, we take the axiomatisation of PDL with positive iteration $+$, rather than with Kleene star $\ast$.
\begin{center}
\begin{tabular}{ll}
\mc{1}{l}{ }          & \mc{1}{l}{\textbf{Box-axioms}}                                                                      \\
K
& $\vdash [\alpha](A\pra B)\pra([\alpha] A\pra[\alpha]B)$                                           \\
Choice               &  $\vdash [\alpha\cup \beta] A\plra [\alpha] A \pand [\beta] A$                                  \\
Composition      & $\vdash [\alpha \gsc \beta]A \plra [\alpha] [\beta] A$                                               \\
Test                    & $\vdash [A \wn ]B \leftrightarrow (A\pra B)$                                                              \\
Distributivity       & $\vdash [\alpha](A\pand B) \plra [\alpha] A \pand [\alpha] B$                                   \\
\hline
Fix Point $\ast$  & $\vdash \ls\alpha^\ast\rs\, A \plra A \pand \ls\alpha\rs\,\ls\alpha^\ast\rs\, A$              \\
Induction $\ast$ & $\vdash A \pand \ls\alpha^\ast\rs\, (A \pra \ls\alpha\rs\, A) \pra \ls\alpha^\ast\rs\, A$ \\
\hline
\hline
Fix Point $+$      & $\vdash [\alpha^+] A \plra [\alpha]A\pand [\alpha][\alpha^+] A$                               \\
Induction $+$     &  $\vdash ([\alpha]A\pand [\alpha^+](A\pra [\alpha] A))\pra[\alpha^+] A$                    \\
\mc{2}{l}{ }                                                                                                                  \\
 \mc{1}{l}{}                              & \mc{1}{l}{\textbf{Inference Rules} }                                   \\
Modus Ponens                       & if $\vdash A \pra B$ and $\vdash A$, then $\vdash B$     \\
$\ls\alpha\rs$-Intro & if $\vdash A$, then $\vdash [\alpha] A$                             \\
\end{tabular}
\end{center}

As was done in \cite{Multitype}, motivated by the fact that the calculus of Section \ref{PDL:sec:language and rules} is modular and can be easily rearranged to take propositional bases which are strictly weaker than the Boolean one, both box and diamond operators will be taken as primitive (see also \cite{Gol} for an axiomatisation of PDL with independent modalities).

\begin{center}
\begin{tabular}{ll}
                          & \textbf{Diamond-axioms} \\
Choice               & $\vdash \lc\alpha \gor \beta\rc\, A \plra \lc\alpha\rc\, A \por \lc\beta\rc\, A$                                                  \\
Composition      & $\vdash \lc\alpha \gsc \beta\rc\, A \plra \lc\alpha\rc \, \lc\beta\rc\, A$                                                           \\
Test                    & $\vdash \lc A\wn\rc\, B \plra A \pand B$                                                                                                      \\
Distributivity       & $\vdash \lc\alpha\rc\, (A \por B) \plra \lc\alpha\rc\, A \por \lc\alpha\rc\, B$                                                   \\
\hline
Fix point $\ast$  & $\vdash \lc\alpha^\ast\rc\, A \plra A \por \lc\alpha\rc \lc\alpha^\ast\rc\, A$                                                     \\
Induction $\ast$ & $\vdash \lc\alpha^\ast\rc\, A \plra A \por \lc\alpha^\ast\rc\, (\neg A \pand \lc\alpha\rc\, A)$                           \\
\hline
\hline
Fix point $+$      & $\vdash \lc\alpha^+\rc\, A \plra \lc\alpha\rc A \por \lc\alpha\rc \lc\alpha^+\rc\, A$                                       \\
Induction $+$     & $\vdash \lc\alpha^+\rc\, A \pra \lc\alpha\rc A \por \lc\alpha^+\rc\, (\neg A \pand \lc\alpha\rc\, A)$                \\
\mc{2}{l}{ }                                                                                                                  \\
 \mc{1}{l}{}                     & \mc{1}{l}{\textbf{Inference Rules} }                                   \\
Modus Ponens              & if $\vdash A \pra B$ and $\vdash A$, then $\vdash B$     \\
$\lc\alpha\rc$-Intro & if $\vdash A \pra \bot$, then $\vdash \lc\alpha\rc A \pra \bot$                             \\

\end{tabular}
\end{center}
The language of PDL  will be also extended with the modalities $\RESalphaBox$ and $\RESalphaDia$ which are adjoint  to $\langle\alpha\rangle$ and $[\alpha]$ respectively for each action $\alpha$.
These modalities correspond to the \textit{converse operator} $ ( \cdot)^{-1}$ with the following semantics: for any action $\alpha$,
$$R_{\alpha^{-1}} :=  R_\alpha^{-1}, $$
where $R^{-1}:= \{ (u,v) \mid (v,u) \in R \}$ for every relation $R$.
As mentioned in \cite{Kozen2000}, the converse operator is useful to talk about `running a program backward' or reversing actions, although this is not always possible in practice.
The modal operators $\RESalphaBox$ and  $\RESalphaDia$ are semantically interpreted in the standard way using the relation $R_{\alpha}^{-1}$.
The relevant additional rules are reported below (see also \cite{Har, Har13} for an analogous extension).

\begin{center}
\begin{tabular}{ll}
\mc{1}{l}{ }  & \mc{1}{l}{\textbf{Inference Rules}}                                                           \\
$(\lc\alpha\rc\dashv \RESalphaBox )$-Adj. & $\fCenter \lc\alpha\rc A \pra B$ \,iff\, $\fCenter A \pra \RESalphaBox B$ \\
$(\RESalphaDia\dashv \ls\alpha\rs)$-Adj.     &
$\fCenter A \pra \ls\alpha\rs B$ \,iff\, $\fCenter \RESalphaDia A \pra B$                      \\
\end{tabular}
\end{center}



\subsection{Multi-type calculi}
\label{ssec:multi}

Here we report on the environment of multi-type display calculi introduced in \cite[Section 3]{Multitype}. 

Our starting point is a propositional language, the terms of which form $n$ pairwise disjoint types $\mathsf{T_1}\ldots \mathsf{T_n}$, each of which with its own signature.
We will use $a, b, c$ and $x, y, z$ to respectively denote operational and structural  terms of  unspecified (possibly different) type. Further, we assume that operational connectives and structural connectives are given {\em both} within each type {\em and} also between different types, so that the display property holds (at least for derivable sequents, cf.\ \cite[Definition 2]{Multitype}, see also below). 

%
%
In the applications we have in mind, the need will arise to support types that are semantically ordered by inclusion.
For example, in Section~\ref{PDL:sec:language and rules} we will introduce, beside the type $\mathsf{Fm}$ of formulas, two types $\mathsf{TAct}$ and $\mathsf{Act}$ of transitive and general actions, respectively. The need for enforcing the distinction between transitive and general actions   in the specific situation of  Section~\ref{PDL:sec:language and rules} arises by the presence of  rules such as {\em absorption} (see table `Iteration Structural Rules' in section \ref{PDL:sec:language and rules}),  which is sound for transitive actions but not for general actions. The semantic point of view suggests to treat $\mathsf{TAct}$ as a proper subset of $\mathsf{Act}$, but our syntactic stipulations, although will be sound w.r.t.\ this state of affairs, will be tuned for the more general situation in which the sets $\mathsf{TAct}$ and $\mathsf{Act}$  are disjoint.
This is convenient as each term can be assigned a unique type {\em unambiguously}. This is a crucial requirement for the Belnap-style cut elimination theorem below, and will be explicitly stated in condition C'$_2$ in the next subsection.

\begin{definition}
\label{PDL:def:type-uniformity}
A sequent $x\vdash y$ is {\em type-uniform} if $x$ and $y$ are of the same type.
\end{definition}
In a display calculus, the cut rule is typically of the following form:
\begin{center}
\AX$X \fCenter A$
\AX$A \fCenter Y$
\RightLabel{$Cut$}
\BI$X \fCenter Y$
\DisplayProof
\end{center}
where $X, Y$ are structures and $A$ is a formula. This translates straightforwardly to the multi-type environment, by the stipulation that  cut rules of the form
\begin{center}
\AX$x \fCenter a$
\AX$a \fCenter y$
\RightLabel{$Cut$}
\BI$x \fCenter y$
\DisplayProof
\end{center}
are allowed in the given multi-type system for each type.

\begin{definition}
\label{PDL:def:strong-type-uniformity}
A cut rule is {\em strongly type-uniform} if its premises and conclusion are of the same type.
\end{definition}

\paragraph{Relativized display property.} The full display property is a key ingredient in the proof of the cut-elimination metatheorem. 

\begin{definition} \label{PDL:def: display prop} (cf.\ \cite[Section 3.2]{Belnap}) 
A proof system enjoys the {\em full display property}  iff for every sequent $X \vdash Y$ and every substructure $Z$  of either  $X$ or  $Y$, the sequent  $X \vdash Y$ can be  transformed, using the rules of the system, into a logically equivalent sequent which is either of the form $Z \vdash W$ or of the form $W \vdash Z$, for some structure $W$. In the first case, $Z$ is \emph{displayed in precedent position}, and in the second case, $Z$ is \emph{displayed in succedent position}.
The rules enabling this equivalent rewriting  are called \emph{display postulates}.
\end{definition}

For instance, it enables a system enjoying it to meet Belnap's condition C$_8$ 
for the cut-elimination metatheorem.
However, it turns out that an analogously good behaviour can be guaranteed of any sequent calculus enjoying the following weaker property:

\begin{definition} \label{PDL:def: relativized display prop}  
A proof system enjoys the {\em relativized display property}  iff for every {\em derivable} sequent $X \vdash Y$ and every substructure $Z$  of either  $X$ or  $Y$, the sequent  $X \vdash Y$ can be  transformed, using the rules of the system, into a logically equivalent sequent which is either of the form $Z \vdash W$ or of the form $W \vdash Z$, for some structure $W$.
\end{definition}
The  calculus defined in Section \ref{PDL:sec:language and rules} does not enjoy the full display property, but does enjoy the relativized display property above, which enables it to verify the condition C'$_8$ (see Section \ref{PDL:ssec:quasi-def}).
The relativized display property is discussed in more details in \cite[Section 2.3]{Multitype}.

\subsection{Proper  multi-type display calculi, and their cut elimination meta\-theorem}
\label{PDL:ssec:quasi-def}

\begin{definition}
\label{PDL:def:quasi-displayable multi-type calculus}
A multi-type display calculus is  {\em proper} 
if it satisfies the following list of conditions:

\paragraph*{C$_1$: preservation of operational terms.} Each operational term occurring in a premise of an inference rule {\em inf}  is a subterm of some operational term in the conclusion of {\em inf}.
\paragraph*{C$_2$: Shape-alikeness of parameters.} Congruent parameters are occurrences of the same structure.\footnote{See \cite[Section 2.2, Condition C$_2$]{GAV}.}
\paragraph*{C'$_2$: Type-alikeness of parameters.}  Congruent parameters have exactly the same type. This condition bans the possibility that a parameter changes type along its history.
\paragraph*{C$_3$: Non-proliferation of parameters.} Each parameter in an inference rule {\em inf} is congruent to at most one constituent in the conclusion of {\em inf}.
\paragraph*{C$_4$: Position-alikeness of parameters.} Congruent parameters are either all antecedent or all succedent parts of their respective sequents.
\paragraph*{C$_5$: Display of principal constituents.} If an operational term $a$ is principal in the conclusion sequent $s$ of a derivation $\pi$, then $a$ is in display.
\paragraph*{C'$_6$: Closure under substitution for succedent parts within each type.} Each rule is closed under simultaneous substitution of arbitrary structures for congruent operational terms occurring in succedent position, {\em within each type}.
\paragraph*{C'$_7$: Closure under substitution for precedent parts within each type.} Each rule is closed under simultaneous substitution of arbitrary structures for congruent operational terms occurring in precedent position, {\em within each type}.

\paragraph*{C'$_8$: Eliminability of matching principal constituents.}

This condition  requests a standard Gentzen-style checking, which is now limited to the case in which  both cut formulas  are {\em principal}, i.e.~each of them has been introduced with the last rule application of each corresponding subdeduction. In this case, analogously to the proof  Gentzen-style, condition C'$_8$ requires being able to transform the given deduction into a deduction with the same conclusion in which either the cut is eliminated altogether, or is transformed in one or more applications of the cut rule, involving proper subterms of the original operational cut-term. In addition to this, specific to the multi-type setting is the requirement that the new application(s) of the cut rule be also {\em strongly type-uniform} (cf.\ condition C$_{10}$ below).

\paragraph*{C$_9$: Type-uniformity of derivable sequents.} Each derivable sequent is type-uniform. 

\paragraph*{C$_{10}$: Strong type-uniformity of cut rules.} All cut rules are strongly type-uniform (cf.\ Definition \ref{PDL:def:strong-type-uniformity}).
\end{definition}

\begin{theorem}
\label{PDL:thm:meta multi}
Any multi-type display calculus satisfying C$_2$, C'$_2$, C$_3$, C$_4$, C$_5$, C'$_6$, C'$_7$, C'$_8$, C$_9$ and C$_{10}$ is cut-admissible. If also C$_1$ is satisfied, then the calculus enjoys the subformula property.
\end{theorem}

The proof of the theorem above is similar to the proof of Theorem
3.3 in \cite{Multitype}.


\section{Language and rules}
\label{PDL:sec:language and rules}
   
\def\mSLA{\mbox{\,\rotatebox[origin=c]{-3.9999}{\TriangleLeft}\raisebox{0.43ex}{$\mkern-1.3mu\thicksim$}\,}}

\def\gSLA{\mbox{\,\gLA\!\!\!\!\!\raisebox{0.1ex}{$\mkern-1.3mu\thicksim$}\,}}

\def\gSRA{\mbox{\,\gRA\!\!\!\!\!\!\!\raisebox{0.1ex}{$\mkern-1.3mu\thicksim$}\,}}

\def\Sprec{\mbox{\,$\prec$
\!\!\!\!
\raisebox{0ex}{$\mkern-1.3mu\thicksim$}\,}}

\def\Ssucc{\mbox{\,
\raisebox{0ex}{$\mkern-1.3mu\thicksim$}
\!\!\!\!
$\succ$\,}}

As mentioned in the introduction, the key idea is to introduce a language in which actions are not  accounted for as {\em parameters} indexing  the dynamic connectives, but as {\em logical terms} in their own right.
In the present section, we define a {\em multi-type language} into which the language of PDL translates, and in which the different types interact via special unary or binary connectives. The present setting consists of the following types: $\mathsf{Act}$ for actions, $\mathsf{TAct}$ for transitive actions, and $\mathsf{Fm}$ for formulas.
We stipulate that  $\mathsf{Act}$,  $\mathsf{TAct}$ and $\mathsf{Fm}$ are pairwise disjoint.

\paragraph*{An algebraically motivated introduction.} Similarly to the binary connectives introduced in \cite{Multitype}, the following binary connectives (referred to as {\em heterogeneous connectives}) facilitate the interaction between the two types of actions and the formulas:
\begin{eqnarray}
{\mand}_0, {\mband}_0& :&  \mathsf{TAct} \times \mathsf{Fm} \to \mathsf{Fm} \\
{\mand}_1, {\mband}_1& :&  \mathsf{Act} \times \mathsf{Fm} \to \mathsf{Fm}.
\end{eqnarray}

\noindent We think of the connectives above as being semantically interpreted as maps preserving existing joins in each coordinate (see below), between algebras suitable to interpret general actions, transitive actions, and formulas respectively. For instance,  suitable   domains of interpretation for formulas can be  complete atomic Boolean algebras or  perfect Heyting algebras; suitable domains of interpretation for actions (in different versions of PDL) can be   quantal frames \cite[Chapter III.2]{resende}, or  Kleene algebras with tests (KATs) \cite[page 421]{Kozen2000},
appropriate subalgebras of which can serve as  domains of interpretation for    transitive actions  (possibly w.r.t.\ to certain restrictions of the signature). 

Connected to the standard relational semantic setting for PDL outlined in Section \ref{PDL:ssec:PDL}, 
for any relational model $M$ based on the set $W$, the complex  algebra based on $\mathcal{P} W$  is taken as the domain of interpretation for  $\mathsf{Fm}$-type terms, and
$\mathsf{Act}$-type (resp.\ $\mathsf{TAct}$-type) terms are interpreted as (transitive) relations on $W$. This way, in particular, for any model $M$, the domain of interpretation of $\mathsf{TAct}$ is the  relation algebra based on the complete lattice $\mathcal{T}(W\times W)$ of the transitive relations on $W$ (indeed,  $\mathcal{T}(W\times W)$ is a sub-$\bigcap$-semilattice of $\mathcal{P} (W\times W)$).

\noindent A natural requirement of the algebraic  environment outlined above, which is verified by the algebras arising from the standard semantic setting of Kripke models,   is that the interpretations of these heterogeneous connectives are {\em actions}, i.e., that (the domains of interpretation of) both $\mathsf{Act}$ and $\mathsf{TAct}$  induce module structures (cf.\ \cite[Chapter II.2]{resende}) on (the domain of interpretation of) $\mathsf{Fm}$. That is, the following conditions hold  for all $\alpha, \beta\in \mathsf{Act}$, $\gamma, \delta\in \mathsf{TAct}$, and $A\in \mathsf{Fm}$,
\begin{eqnarray}
\label{PDL:PDL:eq:action0}
\gamma {\mand}_0(\delta {\mand}_0 A) = (\gamma \gsc \delta){\mand}_0 A &\quad & \alpha {\mand}_1(\beta {\mand}_1 A) = (\alpha \gsc \beta){\mand}_1 A    \\
\label{PDL:PDL:eq:action1}
\gamma {\mband}_{\!0} (\delta {\mband}_{\!0} A) = (\delta \gsc \gamma){\mband}_{\!0} A &\quad & \alpha {\mband}_{\!1} (\beta {\mband}_{\!1} A) = (\beta \gsc \alpha){\mband}_{\!1} A.
\end{eqnarray}

\noindent In  the semantic contexts mentioned above, the fact that the interpretations of the connectives $\mand_i$ and $\mband_{\!i}$ for $i = 0, 1$ are completely join-preserving in both coordinates
implies  that each of them has  right adjoint in each coordinate. In particular,  the following additional connectives have a natural interpretation as the right adjoints of ${\mand}_i$ and $\mband_{\!i}$ for $i = 0, 1$  in their second coordinate:
 \begin{eqnarray}
{\mbra}_{\!0}, {\mra}_{\!0}   & :&  \mathsf{TAct} \times \mathsf{Fm} \to \mathsf{Fm} \\
{\mbra}_{\!1}, {\mra}_{\!1} & :&  \mathsf{Act} \times \mathsf{Fm} \to \mathsf{Fm}.
\end{eqnarray}
\noindent Also, the following connectives can be naturally interpreted in the setting above, as  right adjoints of ${\mand}_1$ and $\mband_{\!1}$  in their first coordinate: 
 \begin{eqnarray}
{\nbla}_{\!1}, {\nla}_{\!1} & : & \mathsf{Fm}\times \mathsf{Fm} \to \mathsf{Act}.
\end{eqnarray}
\noindent  Intuitively, for all formulas $A, B$, the term $B {\nbla}_{\! 1}   A$ denotes the weakest  action $\alpha$ such that, if $A$ was true before $\alpha$ was performed, then $B$ is true after any successful execution of $\alpha$.  The connectives above, when  restricted to the diagonal subset  of  $\mathsf{Fm}  \times \mathsf{Fm}$, are the ones Pratt described as  the {\em weakest preservers} in \cite{Pra91}.
\paragraph*{Virtual adjoints, part 1.} However,  ${\mand}_0$ and $\mband_{\!0}$ cannot be assumed to have right adjoints in their first coordinate (the reason for this will be discussed in part 2 below). Hence, the following connectives cannot be assigned a natural interpretation: 
\begin{eqnarray}
{\msbla}_{\!0}, {\msla}_{\!0} & :&  \mathsf{Fm}  \times \mathsf{Fm} \to \mathsf{TAct}.
\end{eqnarray}
\label{PDL:virtual adjoints} We adopt the following notational convention about the  three different shapes of arrows  introduced so far.
Arrows with  straight tails ($\mra$ and $\mbra$) stand for connectives which have a semantic counterpart and which are included in the language of the Dynamic Calculus for PDL (see the grammar of operational terms on page \pageref{PDL:page:def-formulas});
arrows with no tail  (e.g.\ $\nbla$ and $\nla$) do have a semantic interpretation but are {\em not} included in the language at the operational level, and
arrows with  squiggly tails ($\msbla$ and $\msla$) stand for syntactic objects, called \textit{virtual adjoints}, which do not have a semantic interpretation, but  will play an important role, namely guaranteeing the dynamic calculus to enjoy the relativized display property (cf.\ Definition \ref{PDL:def: relativized display prop}).

%

In what follows, virtual adjoints will be introduced {\em only} as structural connectives. That is,  they will not correspond to any operational connective, and they will not appear actively in any rule schema other than the display postulates (cf.\ Definition \ref{PDL:def: display prop}). 

\noindent The adjunction relations ${\mand}\dashv \mbra, \nbla$ and $\mband\dashv \mra, \nla$ stipulated above translate into the following clauses  for every  action $\alpha$, every transitive action $\delta$,  and all formulas $A$ and $B$:
\begin{eqnarray}
\label{PDL:adjunction:heterogeneous0}
\delta \mand_0 A\leq B\ \mbox{ iff }\ A\leq \delta\mbra_{\!0} B  &\quad&  \delta \mband_{\!0\,} A\leq B\ \mbox{ iff }\ A\leq \delta{\mra}_{\!0} B\\
\alpha \mand_1 A\leq B\ \mbox{ iff }\ A\leq \alpha\mbra_{\!1} B  &\quad&  \alpha \mband_{\!1\,} A\leq B\ \mbox{ iff }\ A\leq \alpha{\mra}_{\!1} B \label{PDL:adjunction:heterogeneous}\\
%
\alpha \mand_1 A\leq B\ \mbox{ iff }\ \alpha\leq B \nbla_{\!1} A  &\quad&  \alpha \mband_{\!1\,} A\leq B\ \mbox{ iff }\ \alpha\leq B{\nla}_{\!1} A.
\label{PDL:adjunction:heterogeneous1-weird}
\end{eqnarray}
\paragraph*{Translating PDL into the multitype language, part 1.}\label{PDL:translation PDL to DC} The intended link between the dynamic connectives of PDL and the multi-type language informally outlined above is illustrated in  Table \ref{PDL:table:translation dynamic modalities}. 
\begin{table}
\begin{center}
\begin{tabular}{c c c c c c c}
$\langle\alpha \rangle A$   & becomes & $\alpha\mand_1 A$ && $\RESalphaDia A$  & becomes & $\alpha\mband_{\!1} A$ \\
 $[\alpha] A$  & becomes & $\alpha\mra_{\!1} A$ && $\RESalphaBox A$  & becomes &  $\alpha\mbra_{\!1} A$ \\
 $\langle\alpha^+ \rangle A$ & becomes & $\alpha^+\mand_0 A$ && $\RESalphaplusDia A$  & becomes &$\alpha^+\mband_{\!0} A$\\
 $[\alpha^+] A$  & becomes &$\alpha^+\mra_{\!0} A$ && $\RESalphaplusBox A$   & becomes &  $\alpha^+\mbra_{\!0} A$. \\
\end{tabular}
\caption{Translating box- and diamond-formulas of PDL into multi-type terms.}
\label{PDL:table:translation dynamic modalities}
\end{center}
\end{table}
\noindent
This table will be extended to account for the disambiguation of the action-only connectives (see further on). This yields  the definition of a formal translation between the language of PDL (possibly extended with adjoints) and that of the Dynamic Calculus. 
We omit the details of this straightforward inductive definition.
In Section \ref{PDL:sec:Soundness}, 
this translation will be elaborated on, and the interpretation of the language of the Dynamic Calculus will be defined so that the translation above preserves the validity of sequents.
In the light of this translation, the adjunction conditions in clauses (\ref{PDL:adjunction:heterogeneous0}) and (\ref{PDL:adjunction:heterogeneous}) correspond to the following  adjunction conditions:
\[
\langle\alpha^+\rangle \dashv \RESalphaplusBox \quad \quad \RESalphaplusDia\dashv[\alpha^+] \quad \quad  \langle\alpha\rangle \dashv \RESalphaBox \quad \quad \RESalphaDia\dashv[\alpha].
\]

\paragraph*{Transitive closure as left adjoint.} 
\label{PDL:Language:para:transitiveClosure}
The other key idea of the design of this calculus is to shape the proof-theoretic  behaviour of the  iteration connective in PDL on the order-theoretic behaviour of the transitive closure. Namely, it is well known (cf.\ \cite[7.28]{DaveyPriestley2002})
that the map associating each binary relation on a given set $W$ with its transitive closure  can be characterized order-theoretically as the left adjoint of the inclusion map 
$\iota: \mathcal{T}(W\times W)\hookrightarrow \mathcal{P}(W\times W)$. Indeed, for every $R\in \mathcal{P}(W\times W)$ and every $T\in \mathcal{T}(W\times W)$, 
\[R^+\subseteq T \ \mbox{ iff }\ R\subseteq \iota(T). \]

This motivates the introduction of {\em two} different types of actions: they are needed in order to  properly express this adjunction.  Thus, we consider the following pair of adjoint maps:
\begin{eqnarray}
(\cdot)^+& :&  \mathsf{Act} \to \mathsf{TAct}\\
(\cdot)^-& :& \mathsf{TAct} \to \mathsf{Act}.
\end{eqnarray}
\noindent The $(\cdot)^+\dashv (\cdot)^-$  adjunction relation stipulated above translates into the following clause  for every  action $\alpha$, and every transitive action $\delta$:
\begin{eqnarray}
\alpha^+\leq \delta &\mbox{ iff }&  \alpha \leq \delta^-.
\label{PDL:adjunction+-}
\end{eqnarray}
\paragraph*{Type-disambiguation of action-parameters.}  We aim at designing a multi-type calculus which verifies condition C'$_2$ about type-alikeness of parameters. The stipulation that $\mathsf{TAct}$ and $\mathsf{Act}$ are disjoint is motivated by this goal, but this alone  is not enough. We also need to introduce several copies of
 sequential composition and non deterministic choice, as follows:
\begin{eqnarray}
\cup_1,\; ;_1  & :& \mathsf{Act}\times \mathsf{Act} \to \mathsf{Act}     \\
\cup_2,\; ;_2  & :& \mathsf{TAct}\times \mathsf{Act} \to \mathsf{Act}   \\
\cup_3,\; ;_3  & :& \mathsf{Act}\times \mathsf{TAct} \to \mathsf{Act}    \\
\cup_4,\;  ;_4 & :& \mathsf{TAct}\times \mathsf{TAct} \to \mathsf{Act} .
\end{eqnarray}
%
\paragraph*{Adjoints for action-connectives.} When actions are interpreted e.g.\ in $\mathcal{P}(W\times W)$ for some set $W$, the natural interpretation of $;_1$   (resp.\ of $\cup_1$) is completely join-preserving (resp.\ meet-preserving) in each coordinate. This implies that their right (resp.\ left) adjoints exist in each coordinate, hence the following connectives have a natural interpretation:
\begin{eqnarray}
  \gdra_1, \gdla_1 & :& \mathsf{Act}\times \mathsf{Act} \to \mathsf{Act}    \\
\gscra_1, \gscla_1 & :& \mathsf{Act}\times \mathsf{Act} \to \mathsf{Act}.
\end{eqnarray}
\noindent Here below, the conditions relative to these adjunctions: for all  $\alpha$, $\beta$, $\gamma\in\mathsf{Act}$,
\begin{eqnarray}
\alpha \leq \beta\cup_1 \gamma\ \mbox{ iff } & \ \beta \gdra_1\alpha \leq \gamma  \ & \mbox{ iff }\ \alpha \gdla_1 \gamma\leq \beta \label{PDL:adjunction:cup1} \\
\alpha \gsc_{\!1\,} \beta\leq \gamma\ \mbox{ iff } & \ \beta \leq \alpha\gscra_1 \gamma  \ & \mbox{ iff }\ \alpha \leq \gamma \gscla_1 \beta. \label{PDL:adjunction:sc1}
\end{eqnarray}
\noindent  Also, residuated operations exist for the $j$-indexed variants, $j\in \{2, 3\}$, of $;$ in their $\mathsf{Act}$-coordinate, and for all $j$-indexed variants of $\cup$ in both coordinates. These operations provide a natural interpretation for the following connectives:

\begin{eqnarray}
  \gdra_2,  \gscra_2 & :& \mathsf{TAct}\times \mathsf{Act} \to \mathsf{Act}    \\
    \gdla_3,  \gscla_3 & :& \mathsf{Act}\times \mathsf{TAct} \to \mathsf{Act}   \\
     \gdla_2,\gdra_3 & :& \mathsf{Act}\times \mathsf{Act} \to \mathsf{TAct}\\
\gdra_4,  & :& \mathsf{TAct}\times \mathsf{Act} \to \mathsf{TAct}\\
\gdla_4 & :& \mathsf{Act}\times \mathsf{TAct} \to \mathsf{TAct}.
\end{eqnarray}

\noindent The adjunction clauses relative to the connectives above  are  analogous to those displayed in (\ref{PDL:adjunction:cup1})--(\ref{PDL:adjunction:sc1}) relative to the connectives of their same shape.

 \paragraph*{Virtual adjoints, part 2.} Since the $\{ 2,3,4\}$-indexed  (resp.\ $0$-indexed) variants of  $;$  (resp.\ of $\mand$ and $\mband$) are to be regarded as  restrictions of their $1$-indexed counterpart, they cannot be assumed to be completely join-preserving  in their $\mathsf{TAct}$-coordinates.  This point is somewhat delicate, so it is worth being expanded on. In the standard semantic setting, the domains of interpretation of $\mathsf{Act}$- and $\mathsf{Tact}$-terms are the algebras $\mathcal{P}(W\times W)$ and $\mathcal{T}(W\times W)$, the domains of which are respectively given by all the binary relations and all the transitive relations on  a given set $W$. As mentioned early on, $\mathcal{T}(W\times W)$ is a sub $\bigcap$-semilattice of $\mathcal{P}(W\times W)$, and hence it is itself a complete lattice. However, for every $\mathcal{X}\subseteq \mathcal{T}(W\times W)$, we have that $\bigvee \mathcal{X}$ in $\mathcal{T}(W\times W)$ coincides with the transitive closure $(\bigcup \mathcal{X})^+$ of $\bigcup \mathcal{X}$. This means in particular that, while  meets in $\mathcal{T}(W\times W)$ coincide with meets in $\mathcal{P}(W\times W)$,  joins in  $\mathcal{T}(W\times W)$ are in general different from  joins in $\mathcal{P}(W\times W)$, or equivalently, $\mathcal{T}(W\times W)$ is not a sub $\bigcup$-semilattice of $\mathcal{P}(W\times W)$. This implies that if the $j$-indexed (resp.\ $0$-indexed) variants of  $;$  (resp.\ of $\mand$ and $\mband$) are to be regarded as  restrictions of their $1$-indexed counterpart, they will preserve the joins of $\mathcal{P}(W\times W)$ but not necessarily those of $\mathcal{T}(W\times W)$. This explains why the $j$-indexed variants of $;\,$ for $j\neq 1$ (resp.\ $\mand_0$ and $\mband_0$) cannot be assumed to be completely join-preserving  in their $\mathsf{TAct}$-coordinates. This implies that they do not have right adjoins in their $\mathsf{TAct}$-coordinates.\footnote{Precisely because meets in $\mathcal{T}(W\times W)$ coincide with meets in $\mathcal{P}(W\times W)$, (the interpretations of) all the $j$-indexed variants of $\cup$ for $2\leq j\leq 4$ are completely meet-preserving in each coordinate, and hence they do have adjoints in each coordinate. A case {\em sui generis} is the one of the connective $\dwn_0$, which denotes the right adjoint of the test operator $\wn_0$ regarded as a map into transitive actions. Notice that, whenever $\mathcal{X}$ is a collection of subsets of the diagonal relation $1_W = \{(z, z)\mid z\in W\}$,   the join of $\mathcal{X}$ in $\mathcal{T}(W\times W)$ does coincide with $\bigcup\mathcal{X}$. Hence, (the interpretation of)  $\wn_0$ is completely join preserving, which implies that  $\dwn_0$ is semantically justified.}
 Hence, the following connectives, which are also referred to as {\em virtual adjoints}, are not semantically justified:
\begin{eqnarray}
\gscla_2,\; \gscra_3  
&:& \mathsf{Act}\times \mathsf{Act} \to \mathsf{TAct}\\
\gscra_4,  & :& \mathsf{TAct}\times \mathsf{Act} \to \mathsf{TAct}\\
\gscla_4 & :& \mathsf{Act}\times \mathsf{TAct} \to \mathsf{TAct}.\\
\nla_0 , \ \nbla_0 & :& \mathsf{Fm}\times \mathsf{Fm} \to \mathsf{TAct}.
\end{eqnarray}
Again, as  discussed in \cite[Section 4]{Multitype}, virtual adjoints are important to guarantee the dynamic calculus for PDL to enjoy the relativized display property, which in turn guarantees the calculus to verify the  condition C'$_8$, crucial to the Belnap-style cut elimination metatheorem. However, to ensure that the virtual adjoints do not add unwanted proof power to the calculus, they will be added to the language only at the structural level, and they will not explicitly interact with any other connective in any rule but in the display rules relative to them.

\paragraph*{Translating PDL into the multitype language, part 2.} To have a complete account of how PDL formulas are to be translated into formulas of the multi-type language, Table \ref{PDL:table: disambiguation} integrates Table \ref{PDL:table:translation dynamic modalities}. 
\begin{table}
\begin{center}
\begin{tabular}{c c c || c c c }
$\alpha \,;\beta $   & $\rightsquigarrow$ & $\alpha \,;_1\beta$ &$\alpha \cup \beta $   & $\rightsquigarrow$ & $\alpha \cup_1\beta$ \\
 $\alpha^+ \,;\beta $   & $\rightsquigarrow$ & $\alpha^+ \,;_2\beta$ & $\alpha^+ \cup\beta $   & $\rightsquigarrow$ & $\alpha^+ \cup_2\beta$  \\
  $\alpha \,;\beta^+ $   & $\rightsquigarrow$ & $\alpha \,;_3\beta^+$ & $\alpha \cup\beta^+ $   & $\rightsquigarrow$ & $\alpha \cup_3\beta^+$  \\
 $\alpha^+ \,;\beta^+ $   & $\rightsquigarrow$ & $\alpha^+ \,;_4\beta^+$ & $\alpha^+ \cup\beta^+ $   & $\rightsquigarrow$ & $\alpha^+ \cup_4\beta^+$  \\
 \end{tabular}
\end{center}
\caption{Type-disambiguation of action-parameters in PDL.}
 \label{PDL:table: disambiguation}
\end{table}

\noindent The different copies of connectives introduced above are needed for the calculus to satisfy condition C'$_2$ about the  type-alikeness of parameters. However, in concrete derivations, as soon as  the type of the atomic constituents is clearly identifiable, the subscripts can be dropped. The disambiguation will also involve the action-type constants, which will be introduced only as structural connectives, but not as operational ones. Specifically,  the structural constants $\Phi_0$ and $\Phi_1$ (when occurring in antecedent position) both correspond to the action $\mathsf{skip}$, regarded as a transitive action or as a general action, respectively. Likewise, the structural constants $\gI_0$ and $\gI_1$, when occurring in antecedent (resp.\ succedent position) both correspond to the action $\mathsf{top}$ (resp.\ $\mathsf{crash}$)   regarded as a transitive action, and  as a general action, respectively.

\paragraph*{Axiomatizing PDL in the multitype language.} Given the translation based on Tables \ref{PDL:table:translation dynamic modalities} and \ref{PDL:table: disambiguation}, the original axioms of PDL can be translated as indicated below.  In what follows, the variables $a, b$ denote terms of type $\mathsf{Act}$ or $\mathsf{TAct}$, the variables $A, B$ denote terms of type $\mathsf{Fm}$ and $\alpha$ is a term of type $\mathsf{Act}$. For every $1\leq j\leq 4$ and  $i = 0, 1$, 

\begin{center}
\begin{tabular}{lrcl}
\mc{4}{l}{\textbf{Box-axioms}} \\
K                       & $\alpha \mra (A \pra B)$ & $\vdash$ & $(\alpha \mra A) \pra (\alpha \mra B)$                                                            \\
Choice              & $(a \gor_j b) \mra A$   & $\dashv\vdash$ & $(a \mra A) \pand (b \mra A)$                                      \\
Composition      & $(a \gsc_j b) \mra A$  & $\dashv\vdash$ & $a \mra (b \mra A)$                                                      \\
Test                   & $A\wn_i \mra B$                      & $\dashv\vdash$ & $A \pra B$                                                                                   \\
Distributivity       & $a \mra (A \pand B)$    & $\dashv\vdash$ & $(a \mra A) \pand (a \mra B)$                                       \\
\hline
Fix point $+$      & $\alpha^+ \mra A$              & $\dashv\vdash$  & $(\alpha \mra A) \pand (\alpha \mra (\alpha^+ \mra A))$                \\
Induction $+$     & $\alpha^+ \mra A$              & $\dashv$             &  $(\alpha \mra A) \pand (\alpha^+ \mra (A \pra (\alpha \mra A)))$  \\

 ~ \\

\mc{4}{l}{\textbf{Diamond-axioms}} \\
Choice               & $(a \gor_j b) \mand A$   & $\dashv\vdash$ & $(a \mand A) \por (b \mand A)$                                                   \\
Composition      & $(a \gsc_j b) \mand A$   & $\dashv\vdash$ & $a \mand (b \mand A)$                                                                \\
Test                    & $A\wn_i \mand B$                       & $\dashv\vdash$ & $A \pand B$                                                                                                \\
Distributivity       & $a \mand (A \por B)$        & $\dashv\vdash$ & $(a \mand A) \por (a \mand B)$                                                    \\
\hline
Fix point $+$      & $\alpha^+ \mand A$                 & $\dashv\vdash$ & $(\alpha \mand A) \por \alpha \mand \alpha^+ \mand A)$                            \\
Induction $+$     & $\alpha^+ \mand A$                 & $\vdash$            & $(\alpha \mand A) \por (\alpha^+ \mand (\neg A \pand (\alpha \mand A)))$ \\
\end{tabular}
\end{center}

\noindent Note that the subscripts of the arrow- and triangle-shaped connectives are completely determined by the type of their arguments in the first coordinate, and hence they have been omitted.

\paragraph*{Additional conditions.} As done and discussed in the setting of \cite{Multitype}, in order to express in the multi-type language that e.g.\ $\langle\alpha\rangle$ and $[\alpha]$ are ``interpreted over the same relation'', Sahlqvist correspondence theory (cf.\ e.g.\ \cite{ALBA, ConPalSou, CGP} for a state-of-the art-treatment) provides us with two alternatives: one of them is that we impose the following Fischer Servi-type conditions \cite{FS84}  to hold for all $a$ of type  $\mathsf{Act}$ or $\mathsf{TAct}$ and $A, B\in \mathsf{Fm}$: for $i = 0, 1$,
\begin{eqnarray*}
( a \mand_i A) \rightarrow (a\mra_{\! i}\ B)  \leq \delta\mra_{\! i}\ (A\rightarrow B)& &( a \mband_{\! i}\ A)\rightarrow (a \mbra_{\! i}\ B)  \leq     a\mbra_{\! i}\ (A\rightarrow B)          \\
 a\mand_i (A \pdra B) \leq ( a \mra_{\! i}\ A) \pdra (a\mand_i B)
& &
a\mband_{\! i}\ (A \pdra B)   \leq ( a \mbra_{\! i}\ A) \pdra (a \mband_{\! i}\ B).
\end{eqnarray*}
To see that the conditions above correspond to the usual Fischer Servi axioms in standard modal languages, one can observe that
 the conditions in the first line above are images, under the translation discussed above, of the Fischer Servi axioms reported on e.g.\ in \cite[Section 6.1]{GAV}.
The second alternative is to impose that, for every $0\leq i\leq 2$, the connectives $\mand_i$ and $\mband_i$ yield  {\em conjugated} diamonds (cf.\ discussion in \cite[Section 6.2]{GAV}); that is, the following inequalities hold for all
 $a$ of type $\mathsf{Act}$ or $\mathsf{TAct}$ and $A, B\in \mathsf{Fm}$:
\begin{eqnarray*}
(a\mand_i A)\wedge B\leq a\mand_i (A\wedge (a \mband_{\! i}\ B)) & & (a\mband_{\! i}\ A)\wedge B\leq a\mband_{\! i}\ (A\wedge (a \mand_i B))\\
a \mra_{\! i}\ (A \vee (a \mbra_{\! i}\ B))
\leq (a\mra_{\! i}\ A) \vee B
 & & a\mbra_{\! i}\ (A \vee ( a \mra_{\! i}\ B))\leq (a\mbra_{\! i}\ A) \vee B.
\end{eqnarray*}
\paragraph*{The operational language, formally.} 
Let us  introduce the operational terms of the multi-type language by the following simultaneous induction, based on sets $\mathsf{AtProp}$ of atomic propositions, and $\mathsf{AtAct}$ of atomic actions: \label{PDL:page:def-formulas}
\begin{align*}
\mathsf{Fm} \ni A:: = \; & p\in \mathsf{AtProp} \mid \bot\mid \top \mid A\wedge A\mid A\vee A\mid A \rightarrow A\mid A \pdra A \mid\\
&  \delta\mand_0 A\mid \delta\mra_{\!0} A  \mid
\alpha\mand_1 A\mid \alpha\mra_{\!1} A  \mid \\
&  \delta\mband_{\!0} A\mid \delta\mbra_{\!0} A \mid
\alpha\mband_{\!1} A\mid \alpha\mbra_{\!1} A
\\
~\\
\mathsf{Act} \ni \alpha:: =  \; & \pi \in \mathsf{AtAct} \mid
\delta^- \mid A \wn_1  \mid \\
& \alpha \gsc_{\!1} \alpha \mid \delta \gsc_{\!2} \alpha \mid \alpha \gsc_{\!3} \delta \mid \delta \gsc_{\!4} \delta \mid\\
&\alpha \gor_1 \alpha \mid \delta \gor_2 \alpha \mid \alpha \gor_3 \delta \mid \delta \gor_4 \delta 
\\
~\\
\mathsf{TAct}\ni \delta:: = \; &  \alpha^+ \mid A \wn_0 
\end{align*}

\paragraph*{Structural language, formally.} Display calculi manipulate two closely related languages: the operational and the structural. Let us introduce the structural language of the Dynamic Calculus, which as usual matches the operational language.
We have {\em formula-type} structures, {\em transitive action-type structures},  {\em action-type structures}, defined by simultaneous recursion as follows: 
\begin{align*}
\mathsf{FM} \ni X:: = \; & A \mid \textrm{I} \mid X \,, X\mid X > X \mid X < X \mid \Pi \DWN_1  \mid \Delta \DWN_0  \mid\\
& \Delta \mAND_{\!0} X\mid \Delta \mRA_{\!0} X \mid \Pi\mAND_{\!1} X\mid \Pi\mRA_{\!1} X  \mid\\
& \Delta \mBAND_{\!0\,} X\mid \Delta \mBRA_{\!0\,} X  \mid \Pi\mBAND_{\!1} X\mid \Pi\mBRA_{\!1} X
\\
~\\
\mathsf{ACT} \ni \Pi:: =\; &
\alpha \mid \,\gI_1\, \mid \Phi_1 \mid \Delta^{\ominus} \mid X \WN_1  \mid \\
& \Pi \gSC_{\!1} \Pi \mid \Pi \succ_1 \Pi \mid \Pi \prec_1 \Pi \mid \Delta \gSC_{\!2} \Pi \mid \Delta \succ_2 \Pi  \mid
 \Pi \gSC_{\!3} \Delta \mid \Pi \prec_3 \Delta \mid \Delta\gSC_{\!4} \Delta \mid\\
& \Pi \gC_{1} \Pi \mid \Pi \gRA_{\!1} \Pi \mid \Pi \gLA_{\!1} \Pi \mid
  \Delta \gC_{2} \Pi \mid \Delta \gRA_{\!2} \Pi \mid \Pi \gC_{3} \Delta \mid \Pi \gLA_{\!3} \Delta \mid \Delta \gC_{4} \Delta  \mid \\
&   X \mLA_{\!1}  X \mid X \mBLA_{\!1} X
\\
~\\
\mathsf{TACT} \ni \Delta:: = \;
& \gI_0\, \mid \Phi_0 \mid \Pi^{\oplus} \mid X \WN_0  \mid  \\
&  \Pi \Sprec_2 \Pi  \mid \Pi \Ssucc_3 \Pi  \mid \Delta \Ssucc_4 \Pi \mid \Pi \Sprec_4 \Delta  \mid \\
&  \Pi \gLA_{\!2} \Pi
\mid \Pi \gRA_{\!3} \Pi
\mid \Delta \gRA_{\!4} \Pi \mid \Pi \gLA_{\!4} \Delta\\
& X \mSLA_{\!0}  X \mid X \mSBLA_{\!0} X
\end{align*}

%
%

\paragraph*{The propositional base.}  As is typical of display calculi, each operational connective corresponds to one structural connective. In particular, the propositional base connectives behave exactly as in \cite{GAV,Multitype}, and their corresponding rules are reported in Appendix \ref{Appendix : The calculus for the propositional base of PDL}.\footnote{The operational connectives $\top, \bot, \pand, \por, \pra$ belong to the language of the most common axiomatizations of propositional classical logic. The operational connectives in brackets $\pdla, \pla, \pdra$ are mentioned in the table for the sake of exhaustiveness. In particular, $\pla$ and $\pra$ (resp.\, $\pdla$ and $\pdra$) are interderivable in the presence of the rule {\em exchange}, and the same is true of the dual connectives $\pdla$ and $\pdra$. The latter two connectives are known as \emph{subtraction} or \emph{disimplication}. The formula $A \pdra B$ (resp.\ $A \pdla B$) is classically equivalent to $\neg A \pand B$ (resp.\ $A \pand \neg B$).}

\begin{center}
\begin{tabular}{|r|c|c|c|c|c|c|c|c|c}
 \hline
 \footnotesize{Structural symbols} & \mc{2}{c|}{I} & \mc{2}{c|}{$,$} & \mc{2}{c|}{$<$}   & \mc{2}{c|}{$>$}     \\
 \hline
 \footnotesize{Operational symbols} & $\top$ &  $\bot$ & $\pand$ & $\por$ & $(\pdla)$ & $(\pla)$ & $(\pdra)$ & $\pra$      \\
\hline
\end{tabular}
\end{center}

\paragraph*{Action connectives, part 1.} As to the $0$-ary and binary action-type connectives
the table below provides the connection between structural and operational connectives for $1 \leq j \leq 4$, 
$h = 1,2$, and
$k = 1,3$. The indexes of the structural connectives are omitted.
\footnote{The operational connectives in brackets are given for the sake of completeness, but they do not belong to the language of the most common axiomatizations of PDL considered here. See \cite{Har, Har13, Pra91} for some extensions of the language and their interpretations.}
\begin{center}
\begin{tabular}{|r|c|c|c|c|c|c|c|c|}
 \hline
 \footnotesize{Structural symbols}    & \mc{2}{c|}{$\gI$}          &  \mc{2}{c|}{$\gC$}                                       & \mc{2}{c|}{$\gRA$}         & \mc{2}{c|}{$\gLA$}         \\
 \hline
 \footnotesize{Operational symbols} & $(\gtop)$ & $(\gbot)$  & \phantom{$\gand$} & $\gor$ & $(\gdra_j)$ & \phantom{$(\gra_h)$}       & $(\gdla_j)$ & \phantom{$(\gla_k)$}        \\
\hline
\mc{9}{c}{}\\
\hline
 \footnotesize{Structural symbols}    &   \mc{2}{c|}{$\Phi$}     &  \mc{2}{c|}{$\gSC$}                                      & \mc{2}{c|}{$\gSCRA$}     & \mc{2}{c|}{$\gSCLA$}    \\
 \hline
 \footnotesize{Operational symbols} & $(1)$ &               & $\gsc$ &                                                      & 
 & $(\gscra_h)$ & 
 & $(\gscla_k)$ \\
 \hline
\end{tabular}

\end{center}

\paragraph*{Heterogeneous connectives.} Similarly to \cite{Multitype}, the heterogeneous structural connectives correspond one-to-one with the operational ones, as illustrated in the following table: for $i =0, 1$,

\begin{center}
\begin{tabular}{| c|c|c | c | c|c|c|}
\hline
\footnotesize{Structural symbols} 
&\mc{2}{c|}{$\mAND_{\!i}$}
&\mc{2}{c|}{ $\mBRA_{\!i} $}
&\mc{2}{c|}{ $\mBLA_{\!1} $}\\
\hline
\footnotesize{Operational symbols} &$\mand_i$&\phantom{$\mand_i$}
&\phantom{$\mra_{\!i}$}& $\mbra_{\!i} $ 
& \phantom{$(\nbla_{\!1})$} & $(\nbla_{\!1})$ \\
\hline
\mc{7}{c}{}\\
\hline
\footnotesize{Structural symbols} 
&\mc{2}{c|}{ $\mBAND_{\!i}$}
&\mc{2}{c|}{$\mRA_{\!i}$}
&\mc{2}{c|}{$\mLA_{\!1}$}\\
\hline
\footnotesize{Operational symbols} 
&$\mband_{\!i}$&\phantom{$\mand_i$}
&\phantom{$\mra_{\!i}$}&$\mra_{\!i}$
& \phantom{$(\nla_{\!1})$} & $(\nla_{\!1})$ \\
\hline
\end{tabular}
\end{center}
That is, the structural connectives are to be interpreted in a context-sensitive way, but the present language  lacks the operational connectives which would correspond to them on one or both of the two sides. This is of course because in the present setting we do not need them.
However, in a setting in which they would turn out to be needed, it would not be difficult to introduce the missing operational connectives.\footnote{See \cite{Har, Har13, Pra91} for some extensions of the language and their interpretations.}
The operational rules for the heterogeneous connectives are essentially the same as the analogous rules given in \cite[Section 4]{Multitype}: in what follows, let $x, y$ and $a$ respectively stand for structural and operational terms of a type which can be either $\mathsf{TAct}$ or  $\mathsf{Act}$,  and  let $Y, Z$ and $B$ respectively stand for structural and operational terms of type $\mathsf{Fm}$; then, for $i = 0,1,$
\begin{center}
\begin{tabular}{rl}
\mc{2}{c}{\textbf{Actions-Propositions Operational Rules}} \\
 & \\
\AX$a \mAND_{\!i} B \fCenter Z$
\LeftLabel{\fns${\mand_i}_L$}
\UI$a \mand_i B \fCenter Z$
\DisplayProof
&
\AX$x \fCenter a$
\AX$Y \fCenter B$
\RightLabel{\fns${\mand_i}_R$}
\BI$x \mAND_{\!i} Y \fCenter a \mand_i B$
\DisplayProof
\\

 & \\

\AX$a \mBAND_{\!i} B \fCenter Z$
\LeftLabel{\fns${\mband_i}_L$}
\UI$a \mband_i B \fCenter Z$
\DisplayProof
&
\AX$x \fCenter a$
\AX$Y \fCenter B$
\RightLabel{\fns${\mband_i}_R$}
\BI$x \mBAND_{\!i} Y \fCenter a \mband_i B$
\DisplayProof
\\

 & \\

\AX$x \fCenter a$
\AX$B \fCenter Y$
\LeftLabel{\fns${\mra_{\!i}}_L$}
\BI$a {\mra_{\!\!i}} B \fCenter x {\mRA_{\!i}} Y$
\DisplayProof
&
\AX$Z \fCenter a \mRA_{\!i} B$
\RightLabel{\fns${\mra_{\!i}}_R$}
\UI$Z \fCenter a \mra_{\!i} B $
\DisplayProof \\

 & \\

\AX$x \fCenter a$
\AX$B \fCenter Y$
\LeftLabel{\fns${\mbra_{\!\!i}}_L$}
\BI$a {\mbra_{\!\!i}} B \fCenter x {\mBRA_{\!\!i}} Y$
\DisplayProof
&
\AX$Z \fCenter a \mBRA_{\!\!i} B$
\RightLabel{\fns${\mbra_{\!\!i}}_R$}
\UI$Z \fCenter a \mbra_{\!\!i} B $
\DisplayProof \\
\end{tabular}
\end{center}
Clearly, the rules above yield the operational rules for the dynamic  modal operators under the translation given early on. Notice that each sequent is always interpreted in one domain; however, since the connectives take arguments of different types (and in this sense we are justified in referring to them as  {\em heterogeneous connectives}), premises of binary rules are of course interpreted in different domains.

\paragraph*{Identity and cut rules.} Axioms  will be given in each type; 
here below,  $\pi \in \mathsf{AtAct}$, and $p\in \mathsf{AtProp}$:

\begin{center}
\begin{tabular}{cc}
\mc{2}{c}{\textbf{Identity Rules}} \\
  & \\
\AXC{\phantom{$a \fCenter a$}}
\LeftLabel{\fns$\pi$ \emph{Id}}
\UI$\pi \fCenter \pi$
\DisplayProof
\quad
&
\quad
\AXC{\phantom{$p \fCenter p$}}
\LeftLabel{\fns$p$ \emph{Id}}
\UI$p \fCenter p$
\DisplayProof
\\
\end{tabular}
\end{center}
\noindent where  the first axiom is of type  $\mathsf{Act}$, and the second one is of type $\mathsf{Fm}$. 

\noindent Further, we allow the following {\em strongly type-uniform} cut rules on the operational terms:

\begin{center}
\begin{tabular}{ccc}
\mc{3}{c}{\textbf{Cut Rules}} \\
 & & \\
\AX$\Gamma \fCenter \delta$
\AX$\delta \fCenter \Delta$
\RightLabel{\fns$\delta$ \emph{Cut}}
\BI$\Gamma \fCenter \Delta$
\DisplayProof
 &
\AX$\Pi\fCenter \alpha$
\AX$\alpha \fCenter \Sigma$
\RightLabel{\fns$\alpha$ \emph{Cut}}
\BI$\Pi \fCenter \Sigma$
\DisplayProof
 &
\AX$X\fCenter A$
\AX$A \fCenter Y$
\RightLabel{\fns$A$ \emph{Cut}}
\BI$X \fCenter Y$
\DisplayProof
\\
\end{tabular}
\end{center}

\paragraph*{Display postulates for  heterogeneous connectives.} Recall that $x$ is a structural variable of type  $\mathsf{TAct}$ or $\mathsf{Act}$, and $Y$ and $Z$ are structural variables of type $\mathsf{Fm}$; for $i = 0,1,$

\begin{center}
\begin{tabular}{rcl}
\mc{3}{c}{\textbf{Actions-Propositions  Display Postulates}} \\
 & \\
\AX$x \mAND_{\!i\,} Y\fCenter Z$
\LeftLabel{\fns{$\mand_i \blacktriangleright_{i}$}}
\doubleLine
\UI$Y \fCenter x {\mBRA_{\!i\,}} Z$
\DisplayProof
 &&
\AX$x \mBAND_i Y\fCenter Z$
\RightLabel{\fns{$\mband_i \vartriangleright_{i}$}}
\doubleLine
\UI$Y \fCenter x {\mRA_{\!i\,}} Z$
\DisplayProof
\\

 && \\

\AX$\pi \mAND_{\!1\,} Y\fCenter Z$
\LeftLabel{\fns{$\mand_1 \blacktriangleleft_{1}$}}
\doubleLine
\UI$\pi \fCenter Z {\mBLA_{\!1\,}} Y$
\DisplayProof
&&
\AX$\pi \mBAND_1 Y\fCenter Z$
\RightLabel{\fns{$\mband_1 \vartriangleleft_{1}$}}
\doubleLine
\UI$\pi \fCenter Z {\mLA_{\!1\,}} Y$
\DisplayProof
\\
 && \\

\AX$\delta \mAND_{\!0\,} Y\fCenter Z$
\LeftLabel{\fns{$\mand_0 \blacktriangleleft_{0}$}}
\doubleLine
\UI$\delta \fCenter Z {\mSBLA_{\!0\,}} Y$
\DisplayProof
&&
\AX$\delta \mBAND_0 Y\fCenter Z$
\RightLabel{\fns{$\mband_0 \vartriangleleft_{0}$}}
\doubleLine
\UI$\delta \fCenter Z {\mSLA_{\!0\,}} Y$
\DisplayProof
\\
\end{tabular}
\end{center}

\noindent Notice that sequents occurring in each display postulate above are {\em not} of the same type. However, it is easy to see that the display postulates preserve the type-uniformity (cf.\ Definition \ref{PDL:def:type-uniformity}); that is, if the premise of any instance of a display postulate is a type-uniform sequent, then so is its conclusion.

\paragraph*{Necessitation, Conjugation, Fischer Servi, and Monotonicity rules.}
For $i = 0,1$, 

\begin{center}
\begin{tabular}{rl}
\mc{2}{c}{\textbf{Necessitation Rules}} \\
 & \\

\AX$\textrm{I}\fCenter W$
\LeftLabel{\fns{${nec_i} \mand$}}
\UI$ x \mAND_{\!i\,} \textrm{I} \fCenter W$
\DisplayProof
 &
 \AX$\textrm{I}\fCenter W$
\RightLabel{\fns{${nec_i} \mband$}}
\UI$x \mBAND_{\!i\,} \textrm{I} \fCenter W$
\DisplayProof
 \\
 \end{tabular}
 \end{center}
 The following rules are derivable from the ones above using the display postulates:
\begin{center}
\begin{tabular}{rl}
\AX$W \fCenter \textrm{I} $
\LeftLabel{\fns{${nec_i} \vartriangleright$}}
\UI$ W \fCenter x \mRA_{\!i\,} \textrm{I}$
\DisplayProof
 &
\AX$W \fCenter \textrm{I} $
\RightLabel{\fns{${nec_i} \blacktriangleright$}}
\UI$W \fCenter x \mBRA_{\!i\,} \textrm{I}$
\DisplayProof
 \\
 \end{tabular}
 \end{center}

\begin{center}
\begin{tabular}{rl}
\mc{2}{c}{\textbf{Conjugation Rules}} \\
 & \\
\AX$x \mAND_{\!\!i\,} ((x \mBAND_{\!\!i\,} Y) \,, Z)\fCenter W$
\LeftLabel{\scriptsize{$({conj_i} \mand)$}}
\UI$Y\,, (x\mAND_{\!\!i\,} Z) \fCenter W$
\DisplayProof
&
\AX$W \fCenter x \mRA_{\!\!i\,} ((x \mBRA_{\!\!i\,} Y) \,, Z)$
\RightLabel{\scriptsize{$({conj_i} \mra)$}}
\UI$W \fCenter Y \,, (x\mRA_{\!\!i\,} Z) $
\DisplayProof
\\

&\\

\AX$x \mBAND_{\!\!i\,} ((x \mAND_{\!\!i\,} Y) \,, Z) \fCenter W$
\LeftLabel{\scriptsize{$({conj_i} \mband)$}}
\UI$Y \,, (x\mBAND_{\!\!i\,} Z) \fCenter W$
\DisplayProof
&
\AX$W \fCenter  x \mBRA_{\!\!i\,} ((x \mRA_{\!\!i\,} Y) \,, Z) $
\RightLabel{\scriptsize{$({conj_i} \mbra)$}}
\UI$W \fCenter  Y \,, (x\mBRA_{\!\!i\,} Z) $
\DisplayProof \\
\end{tabular}
 \end{center}
 The rules above are interderivable with
the following rules using the appropriate display postulates: 
\begin{center}
\begin{tabular}{rl}
\mc{2}{c}{\textbf{Fischer-Servi Rules}} \\

\AX$ (x \mRA_{\!i\,} Y) > (x \mAND_{\!i\,} Z) \fCenter W$
\LeftLabel{\fns{${FS_i}\mand$}}
\UI$x \mAND_{\!i\,} (Y > Z) \fCenter W$
\DisplayProof
 &
\AX$W \fCenter (x \mAND_{\!i\,} Y) > (x \mRA_{\!i\,} Z)$
\RightLabel{\fns{${FS_i}\vartriangleright$}}
\UI$W \fCenter x \mRA_{\!i\,} (Y > Z)$
\DisplayProof
 \\

 & \\

\AX$(x \mBRA_{\!i\,} Y) > (x \mBAND_{\!i\,} Z) \fCenter W$
\LeftLabel{\fns{${FS_i}\mband$}}
\UI$x \mBAND_{\!i\,} (Y > Z) \fCenter W$
\DisplayProof
 &
\AX$W \fCenter (x \mBAND_{\!i\,} Y) > (x \mBRA_{\!i\,} Z)$
\RightLabel{\fns{${FS_i}\blacktriangleright$}}
\UI$W \fCenter x \mBRA_{\!i\,} (Y > Z)$
\DisplayProof
 \\
\end{tabular}
\end{center}
The following rules encode the fact that both arrow- and triangle-shaped heterogeneous connectives are order preserving in their second coordinate.
\begin{center}
\begin{tabular}{rl}
\mc{2}{c}{\textbf{Monotonicity Rules}} \\
 & \\
\AX$(x \mAND_{\!i\,} Y) \,, (x \mAND_{\!i\,} Z) \fCenter W $
\LeftLabel{\fns{${mon_i}\mand$}}
\UI$ x \mAND_{\!i\,} (Y \,, Z)  \fCenter W$
\DisplayProof
&
\AX$W\fCenter ( x \mRA_i Y) \,, (x \mRA_{\!i\,} Z)  $
\RightLabel{\fns{${mon_i}\vartriangleright$}}
\UI$W \fCenter x \mRA_{\!i\,} (Y \,, Z)$
\DisplayProof
\\

 & \\

\AX$ ( x \mBAND_{\!i\,} Y) \,, (x \mBAND_{\!i\,} Z)  \fCenter W $
\LeftLabel{\fns{${mon_i}\mband$}}
\UI$ x \mBAND_{\!i\,} (Y \,, Z)  \fCenter W$
\DisplayProof
 &
\AX$ W \fCenter (x \mBRA_i Y) \,, (x \mBRA_{\!i\,} Z)  $
\RightLabel{\fns{${mon_i}\blacktriangleright$}}
\UI$W \fCenter  x \mBRA_{\!i\,} (Y \,, Z)  $
\DisplayProof
\\
\end{tabular}
\end{center}
\paragraph*{Action rules.} 
The following rules encode   conditions \eqref{PDL:PDL:eq:action0} and \eqref{PDL:PDL:eq:action1}. For   $1\leq j \leq 4$, the subscripts for $\mAND, \mBAND, \mRA, \mBRA$ are omitted since they are uniquely determined by $j$, and $x, y$ are structural variables of the suitable action- or transitive action-type.
\begin{center}
\begin{tabular}{rl}
\mc{2}{c}{\textbf{Actions  Rules}} \\
 & \\
\AX$ x \mAND (y\mAND Z)  \fCenter W $
\doubleLine
\LeftLabel{\fns{${act_j}\mand$}}
\UI$ (x\gSC_{\!\!j} y) \mAND  Z  \fCenter W$
\DisplayProof &
\AX$ x \mBAND (y\mBAND Z)  \fCenter W $
\doubleLine
\LeftLabel{\fns{${act_j}\mband$}}
\UI$ (y\gSC_{\!\!j} x) \mBAND  Z  \fCenter W$
\DisplayProof 
\\
\end{tabular}
\end{center}

The following rules are derivable from the ones above using the display postulates:
\begin{center}
\begin{tabular}{rl}
\AX$ W \fCenter x \mRA(y\mRA Z)  $
\doubleLine
\RightLabel{\fns{${act_j}\vartriangleright$}}
\UI$W \fCenter  (x \gSC_{\!\!j} y) \mRA Z  $
\DisplayProof
&
\AX$ W \fCenter x \mBRA(y\mBRA Z)  $
\doubleLine
\RightLabel{\fns{${act_j}\blacktriangleright$}}
\UI$W \fCenter  (y \gSC_{\!\!j} x) \mBRA Z  $
\DisplayProof
\\
\end{tabular}
\end{center}


\paragraph*{Rules for test and iteration.} Also the {\em unary} heterogeneous structural connectives correspond one-to-one with the operational ones, as illustrated in the following table (the indices are omitted):



\begin{center}
\begin{tabular}{|c|c|c|c|c|c|c|}
\hline
\footnotesize{Structural symbols} & \mc{2}{c|}{ $ \WN$} & \mc{2}{c|}{ ${(\cdot)}^{\oplus} $} &\mc{2}{c|}{ ${(\cdot)}^{\ominus} $} \\
\hline
\footnotesize{Operational symbols} & $\wn$ &\phantom{$\wn$} & ${(\cdot)}^{+}$&\phantom{${(\cdot)}^{+}$}&\phantom{${(\cdot)}^{-} $} & ${(\cdot)}^{-} $ \\
\hline
\end{tabular}
\end{center}
\noindent The operational rules for these connectives are given in the table below, where $i = 0, 1$, and $x$ is a structural variable of suitable action-type or  transitive action-type, uniquely determined so as to satisfy type-regularity.

\begin{center}
\begin{tabular}{@{}lcr@{}}
\mc{3}{c}{\textbf{Test and Iteration Operational Rules}} \\
& & \\

\AX$A \WN_i \fCenter x$
\LeftLabel{\fns$\wn^i_L$}
\UI$A \wn_i \fCenter x$
\DisplayProof
\,
\AX$X \fCenter A$
\RightLabel{\fns$\wn^i_R$}
\UI$X \WN_i \fCenter A \wn_i$
\DisplayProof
 &
\AX$\alpha^{\oplus} \fCenter \Delta$
\LeftLabel{\fns$+_L$}
\UI$\alpha^+ \fCenter \Delta$
\DisplayProof
\,
\AX$\Psi \fCenter \alpha$
\RightLabel{\fns$+_R$}
\UI$\Psi^\oplus \fCenter \alpha^{+}$
\DisplayProof
 &
 
\AX$ \delta\fCenter \Delta $
\LeftLabel{\fns$-_L$}
\UI$\delta^- \fCenter \Delta^{\ominus}$
\DisplayProof
\,
\AX$\Psi \fCenter \delta^{\ominus} $
\RightLabel{\fns$-_R$}
\UI$\Psi \fCenter \delta^-$
\DisplayProof
 \\
\end{tabular}
\end{center}

\begin{center}
\begin{tabular}{lcr}
\mc{3}{c}{\textbf{Test and Iteration Display Postulates}} \\

& & \\

\AX$X \WN_i \fCenter x$
\LeftLabel{\fns$\WN\DWN_i$}
\doubleLine
\UI$X \fCenter x \DWN_i$
\DisplayProof

& &

\AX$\Pi^\oplus \fCenter \Delta$
\LeftLabel{\fns$\oplus\ominus$}
\doubleLine
\UI$\Pi \fCenter \Delta^\ominus$
\DisplayProof
\\
\end{tabular}
\end{center}

\begin{center}
\begin{tabular}{rlcrl}
\mc{2}{c}{\textbf{Test Structural Rules}} \\
 & \\
\AX$X \,, Y \fCenter Z$
\doubleLine
\LeftLabel{\scriptsize{$\WN \mand i$}}
\UI$X\WN_i  \mAND_{\!i\,} Y \fCenter Z$
\DisplayProof
 &
 \AX$X \,, Y \fCenter Z$
\doubleLine
\RightLabel{\scriptsize{$\WN \mband i$}}
\UI$Y\WN_i  \mBAND_{\!i\,} X \fCenter Z$
\DisplayProof
\\
\end{tabular}
\end{center}
The following rules are display equivalent to the ones above.
\begin{center}
\begin{tabular}{rlcrl}

\AX$Y \fCenter X > Z$
\doubleLine
\LeftLabel{\scriptsize{$\WN \vartriangleright i$}}
\UI$Y \fCenter X\WN_i \mRA_{\!i\,} Z$
\DisplayProof
 
 &
\AX$Y \fCenter X > Z$
\doubleLine
\RightLabel{\scriptsize{$\WN \vartriangleright i$}}
\UI$Y \fCenter X\WN_i \mBRA_{\!i\,} Z$
\DisplayProof
 \\
\end{tabular}
\end{center}

\paragraph*{Absorption and promotion/demotion rules.}
\begin{center}
\begin{tabular}{rl}
\mc{2}{c}{\textbf{Absorption   Rules}} \\
 & \\

\AX$\Pi \fCenter \Delta^\ominus$
\LeftLabel{\fns{\emph{abs} $1$}}
\AX$\Sigma \fCenter \Delta^\ominus$
\BI$\Pi \gSC_{\!1\,} \Sigma \fCenter \Delta^\ominus$
\DisplayProof
&
\AX$\Gamma \fCenter \Delta$
\RightLabel{\fns{\emph{abs} $4$}}
\AX$\Xi \fCenter \Delta$
\BI$\Gamma \gSC_{\!4\,} \Xi \fCenter \Delta^\ominus$
\DisplayProof
\\

 & \\

\AX$\Gamma \fCenter \Delta$
\LeftLabel{\fns{\emph{abs} $2$}}
\AX$\Sigma \fCenter \Delta^\ominus$
\BI$\Gamma \gSC_{\!2\,} \Sigma \fCenter \Delta^\ominus$
\DisplayProof
&
\AX$\Sigma \fCenter \Delta^\ominus$
\AX$\Gamma \fCenter \Delta$
\RightLabel{\fns{\emph{abs} $3$}}
\BI$\Sigma \gSC_{\!3\,} \Gamma \fCenter \Delta^\ominus$
\DisplayProof
\\
\end{tabular}
\end{center}


\begin{center}
\begin{tabular}{rl}
\mc{2}{c}{\textbf{Promotion/Demotion  Rules}} \\
 & \\
 
 \AX$\Gamma \gSC_{\!2\,} \Sigma \fCenter \Pi$
\doubleLine
\LeftLabel{\fns{\emph{pro/dem}} 2\textbf{;}1}
\UI$\Gamma^{\ominus} \gSC_{\!1\,} \Sigma \fCenter \Pi$
\DisplayProof

&

\AX$\Pi\fCenter\Gamma \gC_{2} \Sigma$
\doubleLine
\RightLabel{\fns{\emph{pro/dem}} 2$\gC$1}
\UI$\Pi\fCenter\Gamma^{\ominus} \gC_{1} \Sigma $
\DisplayProof
\\

&\\

\AX$\Sigma \gSC_{\!3\,} \Gamma \fCenter \Pi$
\doubleLine
\LeftLabel{\fns{\emph{pro/dem}} 3\textbf{;}1}
\UI$\Sigma \gSC_{\!1\,} \Gamma^{\ominus} \fCenter \Pi$
\DisplayProof

&

\AX$\Pi\fCenter\Sigma \gC_{3} \Gamma $
\doubleLine
\RightLabel{\fns{\emph{pro/dem}} 3$\gC$1}
\UI$\Pi\fCenter\Sigma \gC_{1} \Gamma^{\ominus}$
\DisplayProof
\\

 & \\

\AX$\Delta \gSC_{\!4\,} \Gamma \fCenter \Pi$
\doubleLine
\LeftLabel{\fns{\emph{pro/dem}} 4\textbf{;}2}
\UI$\Delta \gSC_{\!2\,} \Gamma^{\ominus} \fCenter \Pi$
\DisplayProof

&

\AX$\Pi\fCenter\Delta \gC_{4} \Gamma $
\doubleLine
\RightLabel{\fns{\emph{pro/dem}} 4$\gC$2}
\UI$\Pi\fCenter\Delta \gC_{2} \Gamma^{\ominus}$
\DisplayProof
\\

&\\

\AX$\Delta \gSC_{\!4\,} \Gamma \fCenter \Pi$
\doubleLine
\LeftLabel{\fns{\emph{pro/dem}} 4\textbf{;}3}
\UI$\Delta^{\ominus} \gSC_{\!3\,} \Gamma \fCenter \Pi$
\DisplayProof

&  

\AX$\Pi\fCenter\Delta \gC_{4} \Gamma $
\doubleLine
\RightLabel{\fns{\emph{pro/dem}} 4$\gC$3}
\UI$\Pi\fCenter\Delta^{\ominus} \gC_{3} \Gamma$
\DisplayProof
\\

\end{tabular}
\end{center}

\begin{center}
\begin{tabular}{rl}
\mc{2}{c}{
\AX$X \WN_{0}  \fCenter \Delta$
\doubleLine
\LeftLabel{\fns{\emph{pro/dem}} $\WN$}
\UI$X \WN_{1} \fCenter \Delta^{\ominus}$
\DisplayProof}
 \\

& \\

\AX$\Pi^\oplus \mAND_{\!0\,} X \fCenter Y$
\LeftLabel{\fns{\emph{dem} $\mand$}}
\UI$\Pi \mAND_{\!1\,} X \fCenter Y$
\DisplayProof
&

\AX$\Pi^\oplus \mBAND_{\!0\,} X \fCenter Y$
\LeftLabel{\fns{\emph{dem} $\mband$}}
\UI$\Pi \mBAND_{\!1\,} X \fCenter Y$
\DisplayProof
\\
\end{tabular}
\end{center}

\noindent Using the rules above and the Display Postulates, the following rules are derivable:

\begin{center}
\begin{tabular}{rl}

\AX$X \fCenter \Pi^\oplus  \mRA_{\!0\,} Y$
\LeftLabel{\fns{\emph{dem} $\vartriangleright$}}
\UI$X  \fCenter \Pi\mRA_{\!1\,} Y$
\DisplayProof
&
\AX$X \fCenter \Pi^\oplus  \mBRA_{\!0\,} Y$
\RightLabel{\fns{\emph{dem} $\blacktriangleright$}}
\UI$X  \fCenter \Pi\mBRA_{\!1\,} Y$
\DisplayProof
\\
&\\

\AX$\Pi^\oplus  \fCenter X\mSLA_{\!0\,} Y$
\LeftLabel{\fns{\emph{dem} $\vartriangleleft$}}
\UI$\Pi  \fCenter X\mLA_{\!1\,} Y$
\DisplayProof
&
\AX$\Pi^\oplus  \fCenter X\mSBLA_{\!0\,} Y$
\RightLabel{\fns{\emph{dem} $\blacktriangleleft$}}
\UI$\Pi  \fCenter X\mBLA_{\!1\,} Y$
\DisplayProof
\\
\end{tabular}
\end{center}
\paragraph*{Fixed point structural rules.}The following rules correspond to the fixed point axioms.

\begin{center}
\begin{tabular}{rl}
\mc{2}{c}{\textbf{Fixed Point Structural Rules}} \\
 & \\

\AXC{$\Pi\mAND_{\!1\,} X\vdash Y$}
\AXC{$(\Pi\gSC_{\!3\,} \Pi^\oplus)\mAND_{\!1\,} X\vdash Y$}
\LeftLabel{\fns{\emph{FP} $\vartriangle$}}
\BIC{$\Pi^\oplus \mAND_{\!0\,} X \vdash Y  $}
\DisplayProof
 &

\AXC{$\Pi\mBAND_{\!1\,} X\vdash Y$}
\AXC{$(\Pi\gSC_{\!3\,} \Pi^\oplus)\mBAND_{\!1\,} X\vdash Y$}
\RightLabel{\fns{\emph{FP} $\blacktriangle$}}
\BIC{$\Pi^\oplus \mBAND_{\!0\,} X \vdash Y  $}
\DisplayProof
\\
\end{tabular}
\end{center}
Using the rules above and the Display Postulates, the following rules are derivable:
\begin{center}
\begin{tabular}{rl}
\AXC{$X \vdash \Pi \mRA_{\!1\,} Y$}
\AXC{$X \vdash (\Pi \gSC_{\!3\,} \Pi^\oplus) \mRA_{\!1\,} Y$}
\LeftLabel{\fns{\emph{FP} $\vartriangleright$}}
\BIC{$X \vdash \Pi^\oplus \mRA_{\!0\,} Y$}
\DisplayProof
&
\AXC{$X \vdash \Pi \mBRA_{\!1\,} Y$}
\AXC{$X \vdash (\Pi \gSC_{\!3\,} \Pi^\oplus) \mBRA_{\!1\,} Y$}
\RightLabel{\fns{\emph{FP} $\blacktriangleright$}}
\BIC{$X \vdash \Pi^\oplus \mBRA_{\!0\,} Y$}
\DisplayProof
\\
&\\
\AXC{$\Pi \vdash Y \mLA_{\!1\,} X$}
\AXC{$(\Pi\gSC_{\!3\,} \Pi^\oplus) \vdash Y \mLA_{\!1\,} X$}
\LeftLabel{\fns{\emph{FP} $\vartriangleleft$}}
\BIC{$\Pi^\oplus \vdash Y \mSLA_{\!0\,} X$}
\DisplayProof
&
\AXC{$\Pi \vdash Y \mBLA_{\!1\,} X$}
\AXC{$(\Pi\gSC_{\!3\,} \Pi^\oplus) \vdash Y \mBLA_{\!1\,} X$}
\RightLabel{\fns{\emph{FP} $\blacktriangleleft$}}
\BIC{$\Pi^\oplus \vdash Y \mSBLA_{\!0\,} X $}
\DisplayProof
\\
\end{tabular}
\end{center}

\noindent The infinitary {\em iteration} rules are given below:
\begin{center}
\begin{tabular}{rl}
\mc{2}{c}{\textbf{Omega-Iteration Structural Rules}} \\
 & \\

\AXC{$\left(
\begin{array}{c|c}
\Pi^{(n)}\mAND_{\!1\,} X \vdash Y               &n \geq 1
\end{array}\right)$}
\LeftLabel{\fns{\emph{$\omega$ $\vartriangle$}}}
\UIC{$\Pi^\oplus \mAND_{\!0\,} X \vdash Y   $}
\DisplayProof
 &

\AXC{$\left(
\begin{array}{c|c}
\Pi^{(n)}\mBAND_{\!1\,} X \vdash Y               &n \geq 1
\end{array}\right)$}
\RightLabel{\fns{\emph{$\omega$ $\blacktriangle$}}}
\UIC{$\Pi^\oplus \mBAND_{\!0\,} X \vdash Y   $}
\DisplayProof
\\
\end{tabular}
\end{center}

\noindent Using the rules above and the Display Postulates, the following rules are derivable:

\begin{center}
\begin{tabular}{cc}
\AXC{$\left(
\begin{array}{c|c}
X\vdash \Pi^{(n)}   \mRA_{\!1\,} Y            &n \geq 1
\end{array}\right)$}
\LeftLabel{\fns{\emph{$\omega$ $\vartriangleright$}}}
\UIC{$X \vdash \Pi^\oplus  \mRA_{\!0\,} Y  $}
\DisplayProof
&
\AXC{$\left(
\begin{array}{c|c}
X\vdash \Pi^{(n)}   \mBRA_{\!1\,} Y            &n \geq 1
\end{array}\right)$}
\RightLabel{\fns{\emph{$\omega$ $\blacktriangleright$}}}
\UIC{$X \vdash \Pi^\oplus  \mBRA_{\!0\,} Y  $}
\DisplayProof

\\
&\\

\AXC{$\left(
\begin{array}{c|c}
\Pi^{(n)}\vdash Y   \mLA_{\!1\,} X            &n \geq 1
\end{array}\right)$}
\LeftLabel{\fns{\emph{$\omega$ $\vartriangleleft$}}}
\UIC{$\Pi^\oplus  \vdash Y \mSLA_{\!0\,} X  $}
\DisplayProof
&
\AXC{$\left(
\begin{array}{c|c}
\Pi^{(n)}\vdash Y   \mBLA_{\!1\,} X            &n \geq 1
\end{array}\right)$}
\RightLabel{\fns{\emph{$\omega$ $\blacktriangleleft$}}}
\UIC{$\Pi^\oplus  \vdash Y \mSBLA_{\!0\,} X  $}
\DisplayProof
\\
\end{tabular}
\end{center}

\paragraph*{Rules for action constants.} For the following  rules, $j = 1, 2 $ and $k = 1, 3$. Moreover,   $x, y, z$ are structural variables of the suitable action- or transitive action-type. The index on $\gI$ is omitted because it is uniquely determined by $j$ and $k$.

\begin{center}
\begin{tabular}{c c c}
\mc{3}{c}{\textbf{$\gI$-Rules}} \\
  & & \\

\AX$x \fCenter y$
\doubleLine
\RightLabel{\fns$\gI^j_{1R}$}
\UI$ x \fCenter \gI\, \gC_j y$
\DisplayProof

&
\AX$\Delta \fCenter \Gamma$
\doubleLine
\RightLabel{\fns$\gI^3_{1R}$}
\UI$\Delta^{\ominus} \fCenter \gI\, \gC_3 \Gamma$
\DisplayProof

&
\AX$\Delta \fCenter \Gamma$
\doubleLine
\RightLabel{\fns$\gI^4_{1R}$}
\UI$\Delta^{\ominus} \fCenter \gI\, \gC_4 \Gamma$
\DisplayProof
\\

&&\\

\AX$x \fCenter y$
\doubleLine
\RightLabel{\fns$\gI^k_{2R}$}
\UI$x \fCenter y \gC_k\, \gI $
\DisplayProof

&
\AX$\Delta \fCenter \Gamma$
\doubleLine
\RightLabel{\fns$\gI^2_{2R}$}
\UI$\Delta^{\ominus} \fCenter \Gamma \gC_2\, \gI$
\DisplayProof

&
\AX$\Delta \fCenter \Gamma$
\doubleLine
\RightLabel{\fns$\gI^4_{2R}$}
\UI$\Delta^{\ominus} \fCenter \Gamma  \gC_4\, \gI$
\DisplayProof
\\
\end{tabular}
\end{center}

\noindent For the following  rules, $j = 1, 2 $ and $k = 1, 3$. Moreover,   $x, y, z$ are structural variables of the suitable action- or transitive action-type. The index on $\Phi$ is omitted because it is uniquely determined by $j$ and $k$.

\begin{center}
\begin{tabular}{c c c}
\mc{3}{c}{\textbf{$\Phi$-Rules}} \\
  & & \\

\AX$x \fCenter y$
\doubleLine
\LeftLabel{\fns$\Phi^j_{1L}$}
\UI$\Phi\, \gSC_{\!\! j} x \fCenter y$
\DisplayProof

&
\AX$\Delta \fCenter \Gamma$
\doubleLine
\LeftLabel{\fns$\Phi^3_{1L}$}
\UI$\Phi\, \gSC_{\!\! 3} \Delta \fCenter \Gamma^{\ominus}$
\DisplayProof

&
\AX$\Delta \fCenter \Gamma$
\doubleLine
\LeftLabel{\fns$\Phi^4_{1L}$}
\UI$\Phi\, \gSC_{\!\! 4} \Delta \fCenter \Gamma^{\ominus}$
\DisplayProof
\\

&&\\

\AX$x \fCenter y$
\doubleLine
\LeftLabel{\fns$\Phi^k_{2L}$}
\UI$x \gSC_{\!\! k}\, \Phi \fCenter y$
\DisplayProof

&
\AX$\Delta \fCenter \Gamma$
\doubleLine
\LeftLabel{\fns$\Phi^2_{2L}$}
\UI$\Delta \gSC_{\!\!2}\, \Phi \fCenter \Gamma^{\ominus}$
\DisplayProof

&
\AX$\Delta \fCenter \Gamma$
\doubleLine
\LeftLabel{\fns$\Phi^4_{2L}$}
\UI$\Delta \gSC_{\!\!4}\, \Phi \fCenter \Gamma^{\ominus}$
\DisplayProof
\\
\end{tabular}
\end{center}

{\commment{
 \noindent For the following  rules, $j = 1, 2 , 5$ and $j_1 = 3, 4$; moreover,  $k = 1, 3, 5$, and $k_1 = 2, 4$; finally, $h = 1, 2$, and $l = 1, 3$. Moreover,   $x, y, z$ are structural variables of the suitable action- or transitive action-type. The index on $\gI$ is omitted because it is uniquely determined by $j$ and $k$.

\begin{center}
\begin{tabular}{c c c}

\mc{2}{c}{\textbf{Weakening Rules for Actions}} \\
 &  \\

\AX$x \fCenter y$
\doubleLine
\LeftLabel{\fns$W^j_{1L}$}
\UI$z\, \gC_j x \fCenter y$
\DisplayProof

&
\AX$\Delta \fCenter \Gamma$
\doubleLine
\LeftLabel{\fns$W^{j_1}_{1L}$}
\UI$z\, \gC_{j_1} \Delta \fCenter \Gamma^{\ominus}$
\DisplayProof
\\

&\\

\AX$x \fCenter y$
\doubleLine
\LeftLabel{\fns$W^k_{2L}$}
\UI$x \gC_k\, z \fCenter y$
\DisplayProof

&
\AX$\Delta \fCenter \Gamma$
\doubleLine
\LeftLabel{\fns$W^{k_1}_{2L}$}
\UI$\Delta \gC_{k_1}\, z \fCenter \Gamma^{\ominus}$
\DisplayProof
\\

&\\

\AX$x \fCenter y$
\doubleLine
\RightLabel{\fns$W^h_{1R}$}
\UI$x \fCenter z \gC_h y$
\DisplayProof

&
\AX$x \fCenter y$
\doubleLine
\RightLabel{\fns$W^l_{2R}$}
\UI$x \fCenter y \gC_l z$
\DisplayProof

\\

\end{tabular}
\end{center}
}}

\paragraph*{Structural rules for binary action connectives.} For the following  rules (cf.\ \cite{Har, Har13} for Weakening w.r.t. sequential composition), $j = 1, 2 $ and $k = 1, 3$. Moreover,   $x, y, z$ are structural variables of the suitable action- or transitive action-type.

\begin{center}
\begin{tabular}{c c c}
\mc{3}{c}{\textbf{Weakening Rules for Actions}} \\
 & & \\
 
\AX$x \fCenter y$
\RightLabel{\fns$W^h_{1R}$}
\UI$x \fCenter z \gC_j   y$
\DisplayProof

&
\AX$\Delta \fCenter \Gamma$
\RightLabel{\fns$W^3_{1R}$}
\UI$\Delta^{\ominus} \fCenter \Pi \gC_3 \Gamma$
\DisplayProof

&
\AX$\Delta \fCenter \Gamma$
\RightLabel{\fns$W^4_{1R}$}
\UI$\Delta^{\ominus} \fCenter \Gamma' \gC_4 \Gamma$
\DisplayProof
\\

&&\\

\AX$x \fCenter y$
\RightLabel{\fns$W^l_{2R}$}
\UI$x \fCenter y \gC_k z$
\DisplayProof

&
\AX$\Delta \fCenter \Gamma$
\RightLabel{\fns$W^2_{2R}$}
\UI$\Delta^{\ominus}  \fCenter \Gamma \gC_2 \Pi $
\DisplayProof

&
\AX$\Delta \fCenter \Gamma$
\RightLabel{\fns$W^4_{2R}$}
\UI$\Delta^{\ominus}  \fCenter \Gamma \gC_4 \Gamma'$
\DisplayProof
\\
\end{tabular}
\end{center}



\noindent For the following  rules, $k=1,4$   and $x, y, z$ are structural variables of the suitable action- or transitive action-type required by type-regularity.
\begin{center}
\begin{tabular}{c}
\mc{1}{c}{\textbf{Contraction Rule for Actions}} \\
 \\
\AX$y \fCenter x \gC_k x$
\RightLabel{\fns$C^k_R$}
\UI$y \fCenter x$
\DisplayProof
 \\
\end{tabular}
\end{center}

\noindent Additional {\em contraction} rules can be derived using the {\em promotion/demotion} rules. For the following  rules, $k = 1, 4$.  Moreover, $x, y, z$ are structural variables of the suitable action- or transitive action-type required by type-regularity.
\begin{center}
\begin{tabular}{lcr}
\mc{3}{c}{\textbf{Exchange Rules for Actions}} \\
 & &  \\

\AX$ \Sigma \fCenter\Delta \gC_2\Pi $
\RightLabel{\fns$E^{2\gC 3}_R$}
\doubleLine
\UI$\Sigma \fCenter\Pi \gC_3 \Delta$
\DisplayProof
 & 
 &
\AX$z \fCenter x \gC_k y$
\RightLabel{\fns$E^{k\gC k}_R$}
\UI$z \fCenter y \gC_k x$
\DisplayProof
 \\
\end{tabular}
\end{center}
\noindent 
For the following  rules, the indices are omitted, under the convention that they span over all the combinations allowed by the grammar, by type-regularity and by type-alikeness of parameters. The variables $x, y, z$ are  of the suitable action- or transitive action-type. 
\begin{center}
\begin{tabular}{lcr}
\mc{3}{c}{\textbf{Associativity Rules for Actions}} \\

 & & \\
 
\AX$x \gSC (y \gSC z) \fCenter w$
\LeftLabel{\fns$A_{L}$}
\UI$(x \gSC y) \gSC z \fCenter w$
\DisplayProof
 &&
\AX$w \fCenter (z \gC y) \gC x$
\RightLabel{\fns$A_{R}$}
\UI$w \fCenter z \gC (y \gC x)$
\DisplayProof
\\
\end{tabular}
\end{center}

\paragraph*{Dynamics and non-deterministic choice.} In the following {\em choice} rules, the index on $\gC$ is  uniquely determined and is omitted. These rules encode the fact that 
$\blacktriangleleft_1$ and $\vartriangleleft_1$ are monotone in their first coordinate and antitone in their second coordinate.

\begin{center}
\begin{tabular}{rl}

\mc{2}{c}{\textbf{Structural Rules for Non-Deterministic Choice}} \\
 & \\
\AX$\Psi \fCenter (Y \mBLA_{\!\!1} X) \gC (Z \mBLA_{\!\!1} X)$
\LeftLabel{\fns{$choice \blacktriangleleft_1$}}
\UI$\Psi \fCenter (Y \,, Z) \mBLA_{\!\!1} X$
\DisplayProof
 &
 \AX$\Psi \fCenter (Y \mLA_{\!\!1} X) \gC (Z \mLA_{\!\!1} X)$
\RightLabel{\fns{$choice \vartriangleleft_1$}}
\UI$\Psi \fCenter (Y \,, Z) \mLA_{\!\!1} X$
\DisplayProof
\\

&\\
\AX$\Psi \fCenter (X \mBLA_{\!\!1} Y) \gC (X \mBLA_{\!\!1} Z)$
\LeftLabel{\fns{$choice \blacktriangleleft^1$}}
\UI$\Psi \fCenter X \mBLA_{\!\!1} (Y \,, Z)$
\DisplayProof
&
\AX$\Psi \fCenter (X \mLA_{\!\!1} Y) \gC (X \mLA_{\!\!1} Z)$
\RightLabel{\fns{$choice \vartriangleleft^1$}}
\UI$\Psi \fCenter X \mLA_{\!\!1} (Y \,, Z)$
\DisplayProof
 \\
\end{tabular}
\end{center}

\paragraph*{More rules on non-deterministic choice and sequential composition.} 
For the following  rules, $1\leq k\leq 4$. Moreover,    $x, y, z$ are structural variables of the suitable action- or transitive action-type.

\begin{center}
\begin{tabular}{rlcrl}
\mc{5}{c}{\textbf{Display Postulates for Non-Deterministic Choice and Sequential Composition}} \\
 &&&& \\

\AX$z \fCenter x \gC_k y$
\doubleLine
\UI$x \gRA_{\!\! k} z \fCenter y$
\DisplayProof
 &
\AX$z \fCenter x \gC_k y$
\doubleLine
\UI$z \gLA_{\!\! k} y \fCenter x$
\DisplayProof
&&

\AX$x \gSC_{\!\! j} y \fCenter z$
\doubleLine
\UI$y \fCenter x \succ_{ j} z$
\DisplayProof
 &
\AX$x \gSC_{\!\! j} y \fCenter z$
\doubleLine
\UI$x \fCenter z \prec_{ j} y$
\DisplayProof
 \\
\end{tabular}
\end{center}


\noindent Finally, the rules for the operational connectives $ \cup_j$ and $;_j$ are given below. For the following rules,  $1\leq j\leq 4$, and the variables $x, y$ and $f, g$ respectively denote structural and operational terms of suitable type uniquely determined by $j$ and by term-uniformity.
\begin{center}
\begin{tabular}{rlrl}
\mc{4}{c}{\textbf{ Operational Rules for Non-Deterministic Choice and Sequential Composition}} \\
 & & & \\

\AX$f \fCenter x$
\AX$g \fCenter y$
\LeftLabel{\fns$\gor^j_L$}
\BI$f \gor_{ j} g \fCenter x \gC_{ j} y$
\DisplayProof
 & 
 \AX$x \fCenter f \gC_{j} g$
\RightLabel{\fns$\gor^j_R$}
\UI$x \fCenter f \gor_{j} g$
\DisplayProof

  &

\AX$ f \gSC_{\!\! j} g \fCenter x$
\LeftLabel{\fns$\gsc^j_L$}
\UI$f \gsc_{\!\! j} g \fCenter x$
\DisplayProof
 &
\AX$x \fCenter f$
\AX$y \fCenter g$
\RightLabel{\fns$\gsc^j_R$}
\BI$x \gSC_{\!\! j} y \fCenter f \gsc_{\!\! j} g$
\DisplayProof

 \\

\end{tabular}
\end{center}

{}

\section{Soundness}
\label{PDL:sec:Soundness}
   
\newcommand{\lsem}{\mathopen{[\![}}
\newcommand{\rsem}{\mathclose{]\!]}}
\newcommand{\sem}[1]{\lsem #1 \rsem}

In the present section, we discuss the soundness of the rules of the dynamic calculus and  prove that those which do not involve virtual adjoints (cf.\ Section \ref{PDL:sec:language and rules})
are sound with respect to the standard relational semantics. As we will see, the interpretation of the multi-type language which we are about to define preserves the translation from the standard PDL language to the multi-type one, which was outlined in Tables \ref{PDL:table:translation dynamic modalities} and  \ref{PDL:table: disambiguation}.

A {\em model}  for the multi-type language for PDL is a tuple $N = (W, v)$ such that $W$ is a nonempty set, and $v$ is a variable assignment from $\mathsf{AtProp}\cup \mathsf{AtAct}$ mapping each $p\in \mathsf{AtProp}$ to a subset 
$\sem{p}_V\subseteq W$, 
and each $\pi\in \mathsf{AtAct}$ to a binary relation $R_{\pi}\subseteq W\times W$. Clearly, these models bijectively correspond to standard Kripke models for PDL: indeed, for every standard Kripke model $M = (W, \mathcal{R}, V)$ such that $\mathcal{R} = \{R_\pi\mid \pi\in \mathsf{AtAct}\}$, let $N_M: = (W, v_M)$, where $v_M(p) = V(p)$ 
for every $p\in \mathsf{AtProp}$, and $v_M(\pi) = R_\pi$ for every $\pi\in \mathsf{AtAct}$. Conversely, for every $N = (W, v)$ as above, let $M_N: = (W, \mathcal{R}_N, V_N)$ such that $\mathcal{R}_N: = \{v(\pi)\mid \pi\in \mathsf{AtAct}\}$, and $V_N(p) = v(p)$ for every $p\in \mathsf{AtProp}$. It is immediate to verify that $N_{M_N} = N$ and $M_{N_M} = M$ for every $M$ and $N$ as above. Clearly each model $N$ as above gives rise to algebras $\mathcal{P}(W)$, $\mathcal{P}(W\times W)$ and $\mathcal{T}(W\times W)$, which provide suitable domains of interpretations of terms of type $\mathsf{Fm}$, $\mathsf{Act}$ and $\mathsf{TAct}$, respectively.

Structures will be translated into operational terms of the appropriate type, and operational terms  will be interpreted
according to their type. 
In order to translate structures as operational terms, structural connectives need to be translated as logical connectives. To this effect,
non-modal, propositional structural connectives are associated with pairs of logical connectives, and any given occurrence of a structural connective is translated as one or the other, according to
its (antecedent or succedent) position.
The following table illustrates how to translate each propositional structural connective of type $\mathsf{FM}$, in the upper row, into one or the other of the logical connectives corresponding to it on the lower row: the one on the left-hand (resp.\ right-hand) side, if the structural connective occurs in precedent (resp.\ succedent) position.

%
\begin{center}
\begin{tabular}{|r|c|c|c|c|c|c|c|c|c}
 \hline
 \footnotesize{Structural symbols} & \mc{2}{c|}{$<$}   & \mc{2}{c|}{$>$} & \mc{2}{c|}{$;$} & \mc{2}{c|}{I}    \\
 \hline
 \footnotesize{Operational symbols} & $\pdla$ & $\leftarrow$ & $\pdra$ & $\rightarrow$  & $\wedge$ & $\vee$  & $\top$ & $\bot$    \\
\hline
\end{tabular}
\end{center}
Recall that, in the Boolean setting treated here, the  connectives $\pdla$ and $\pdra$ are interpreted as $A\pdla B : = A\wedge \neg B$ and $A\pdra B : = \neg A\wedge B$.
The soundness of structural and operational rules which only involve active components of type $\mathsf{FM}$
has been discussed in \cite{GAV}  and is here therefore omitted.

\vspace{5px}

The following table illustrates, with the reading indicated above, how to translate each action-type structural connective. 
Notice that some of the operational connectives in the table below are not included in the operational language of the dynamic calculus for PDL. However, as discussed in Section \ref{PDL:sec:language and rules}, 
the operational symbols below are the ones endowed with a semantic justification (so although the indexes are omitted, it is understood that the table below refers to no virtual adjoints). So even if they are not included in the language, they are used in the present section to facilitate the semantic interpretation of structures occurring in sequents. Notice also that  the structural connectives below have a semantic interpretation only when occurring in precedent (resp.\ succedent) position. Hence, not every structure is going to be semantically interpretable.  However, as we will see, this  is enough for checking the soundness of the rules. 
\begin{center}
\begin{tabular}{|r|c|c|c|c|c|c|c|c|}
 \hline
 \footnotesize{Structural symbols}    & \mc{2}{c|}{$\gI$}          &  \mc{2}{c|}{$\gC$}                                       & \mc{2}{c|}{$\gRA$}         & \mc{2}{c|}{$\gLA$}         \\
 \hline
 \footnotesize{Operational symbols} &  & $\gbot$  &  & $\gor$ & $\gdra$ &        & $\gdla$ &         \\
\hline
\mc{9}{c}{}\\
\hline
 \footnotesize{Structural symbols}    &   \mc{2}{c|}{$\Phi$}     &  \mc{2}{c|}{$\gSC$}                                      & \mc{2}{c|}{$\gSCRA$}     & \mc{2}{c|}{$\gSCLA$}    \\
 \hline
 \footnotesize{Operational symbols} & $1$ &               & $\gsc$ &                                                      & 
 & $\gscra$ & 
 & $\gscla$ \\
 \hline
\mc{9}{c}{}\\
 \hline
 \footnotesize{Structural symbols}                       & \mc{2}{c|}{ $ \WN$} & \mc{2}{c|}{ $ \DWN$}& \mc{2}{c|}{ ${(\cdot)}^{\oplus} $} &\mc{2}{c|}{ ${(\cdot)}^{\ominus} $}      \\
 \hline
 \footnotesize{Operational symbols} &  $\wn$ &\phantom{$\wn$} & \phantom{$\dwn$}   &$\dwn$& ${(\cdot)}^{+}$&\phantom{${(\cdot)}^{+}$}&   \phantom{${(\cdot)}^{-} $} & ${(\cdot)}^{-} $       \\
\hline
\end{tabular}
\end{center}

The interpretation of the  connectives above corresponds to the standard one discussed in Sections \ref{PDL:sec:Preliminaries} and \ref{PDL:sec:language and rules}.  
Below, $a$ and $b$ are operational terms of type $\mathsf{Act}$ or $\mathsf{TAct}$, $\alpha$,  $\delta$ and $A$ are operational terms of type $\mathsf{Act}$, $\mathsf{TAct}$ and $\mathsf{Fm}$ respectively, and the indexes are omitted.
\begin{align*}
\sem{a \cup b }_v&=\{(z, z')\in W\times W\mid (z, z')\in \sem{a}_v\, \mbox{ or }\, (z, z')\in\sem{b}_v\}
& \\
\sem{a \ ;\ b}_v&=\{(z, z')\in W\times W\mid \exists w \,.\, (z, w)\in \sem{a}_v \ \&\  (w, z')\in\sem{b}_v\}
& \\
\sem{\gbot }_v&=\{(z, z')\in W\times W\mid (z, z')\neq (z, z')\} = \varnothing
& \\
\sem{1}_v&=\{(z, z)\in W\times W\mid z\in W\}
& \\
\sem{ A \wn}_v&=\{(z, z)\in W\times W\mid   z\in\sem{A}_v\}
&  \\
\sem{ a \dwn }_v&=\{z\in W\mid (z, z)\in \sem{a}_v\}
&  \\
\sem{\alpha^+}_v&=\bigcup_{n\geq 1}\sem{\alpha}_v^n
& \\
\sem{\delta^-}_v&=\sem{\delta}_v
& \\
\sem{a \gscra b }_v&=\{(z, z')\in W\times W\mid \forall w\,.\,((w, z)\in \sem{a}_v\, \Rightarrow \, (w, z')\in\sem{b}_v)\}
& \\
\sem{a \gscla\ b}_v&=\{(z, z')\in W\times W\mid \forall w \,.\,( (z', w)\in \sem{b}_v \, \Rightarrow\,  (z, w)\in\sem{a}_v)\}
& \\
\sem{a \gdra b }_v&=\{(z, z')\in W\times W\mid (z, z')\in \sem{b}_v\, \ \&\  \, (z, z')\notin\sem{a}_v\}
  = \sem{b \gdla\ a}_v & 
\end{align*}
Given this standard interpretation, the verification of the soundness of  the pure-action rules is straightforward, and is omitted. 

As to the heterogeneous connectives, their translation into the corresponding operational connectives is indicated in the table below,
to be understood similarly to the one above,
where the index $i$ ranges over $\{0,1\}$. 
%
\begin{center}
\begin{tabular}{| c|c|c | c | c|c|c|c|c|c|c|c|c|}
\hline
\footnotesize{Structural symbols} &\mc{2}{c|}{$\mAND_{\!i}$}&\mc{2}{c|}{ $\mBAND_{\!i}$}&\mc{2}{c|}{$\mRA_{\!i}$}&\mc{2}{c|}{ $\mBRA_{\!i} $}&\mc{2}{c|}{$\mLA_{\!1}$}&\mc{2}{c|}{ $\mBLA_{\!1} $}\\
\hline
\footnotesize{Operational symbols} &$\mand_i$&\phantom{$\mand_i$}& $\mband_{\!i}$&\phantom{$\mand_i$}&\phantom{$\mra_{\!i}$}&$\mra_{\!i}$&\phantom{$\mra_{\!i}$}& $\mbra_{\!i} $ &\phantom{$\nla_{\!1}$}&$\nla_{\!1}$&\phantom{$\nla_{\!1}$}& $\nbla_{\!1} $\\
\hline
\end{tabular}
\end{center}


The interpretation of the heterogeneous connectives involving formulas and actions corresponds to that of the well known forward and backward modalities discussed in Section \ref{PDL:sec:language and rules}
(below on the right-hand side we recall the notation in the standard language of PDL with adjoint modalities):
\begin{align*}
\sem{\alpha \mand_1 A}&=\{z\in W\mid \exists z' \,.\, z\, R_\alpha\, z' \ \& \ z'\in\sem{A}\}
& \langle\alpha\rangle A\\
\sem{\alpha \mband_{\! 1} \ A}&=\{z\in W\mid \exists z \,.\, z'R_\alpha\, z \ \&\  z'\in\sem{A}\}
& \RESalphaDia A\\
\sem{\alpha \mra_{\!\! 1} \ A}&=\{z\in W\mid \forall z' \,.\, z\,R_\alpha\, z' \Rightarrow z'\in\sem{A}\}
& [\alpha] A \\
\sem{\alpha \mbra_{\!\! 1} \ A}&=\{z\in W\mid \forall z \,.\, z'R_\alpha\, z \Rightarrow z'\in\sem{A}\}
& \RESalphaBox A
\end{align*}
The connectives $\mand_0,\mra_{\!\! 0},\mband_{\! 0},\mbra_{\!\! 0}$, involving formulas and transitive actions, are interpreted in the same way, replacing the relation $R_{\alpha}$ with the appropriate transitive relations $R_\delta$.
Finally, the following syntactic adjoints can be given an interpretation as follows:
\begin{align*}
\sem{B \nla_1 A}_v&=\{(z, z')\in W\times W\mid z \in \sem{A}_v\, \Rightarrow \,  z'\in\sem{B}_v\}
& \\
\sem{B \nbla_1 A}_v&=\{(z, z')\in W\times W\mid  z'\in \sem{A}_v \, \Rightarrow\,  z\in\sem{B}_v\}
&
\end{align*}

It can also be readily verified that the translation of Section \ref{PDL:sec:language and rules} preserves the semantic interpretation, that is, 
 $\sem{A}_M = \sem{A'}_{N_M}$ for every Kripke model $M$ and any PDL-formula $A$, where $A'$ denotes the translation of $A$ in the language of the dynamic calculus.


The soundness of  all operational rules for  heterogeneous connectives
immediately follows from the fact that their semantic
counterparts as defined above are monotone or antitone
in each coordinate.

The soundness of the cut-rules follows from the transitivity of the inclusion relation in the domain of interpretation of each type.

The display rules $(\mand_i , \mbra_{\!\!i})$ and $(\mband_i , \mra_{\!\!i})$ for $0\leq i\leq 1,$ and
$(\mand_1 , \nbla_{\!\!1})$ and $(\mband_{\! 1} , \nla_{\!\!1})$
are sound as the semantics of the triangle and arrow connectives form adjoint pairs.

On the other hand, in the display rules    $(\mand_0 , \msbla_{\!\!0})$ and $(\mband_0 , \msla_{\!\!0})$, the arrow-connectives are what we  call \emph{virtual adjoints} (cf.\ Section \ref{PDL:sec:language and rules}),
that is, they do not have a semantic interpretation.
In the next section, we will discuss a proof method to show that their presence in the calculus is safe.\footnote{At the moment, this is still a conjecture.}
Soundness of  necessitation, conjugation, Fischer Servi, and monotonicity rules is straightforward and proved as in \cite[Section 6.2]{GAV}.
In the remainder of the section, we discuss the soundness of the fixed point and omega-rules.
As to the soundness of \textit{FP}$\mand$, fix a model $N = (W, v)$, assume that the  structures $X$, $Y$ and $\Pi$ have been assigned  interpretations, denoted (abusing notation) $\sem{X}_v$, $\sem{Y}_v\subseteq W$ and $R = \sem{\Pi}_v\subseteq W\times W$,  and that the premises of $FP\mand$ are satisfied, that is:
\[ R^{-1}[\sem{X}_v]\subseteq \sem{Y}_v\quad \mbox{ and }\quad  (R\circ R^+)^{-1}[\sem{X}_v]\subseteq \sem{Y}_v.\]
We need to show that $(R^+)^{-1}[\sem{X}_v]\subseteq \sem{Y}_v$. By definition, $R^+ = \bigcup_{n\geq 1}R^n$, where $R^1 = R$, and $R^{n+1} = R\circ R^{n}$. Hence, $(R^+)^{-1}[\sem{X}_v] = \bigcup_{n\geq 1} (R^{n})^{-1}[\sem{X}_v]$. Therefore it is enough to show that, for every $n\geq 1$,
\[(R^{n})^{-1}[\sem{X}_v]\subseteq \sem{Y}_v.\]
This is shown by induction on $n$. Both  the base and the induction  cases follow by the assumptions. The soundness of the remaining $FP$-rules is shown similarly, and is omitted.

As to the soundness of the rule $\omega\!\mand$, fix $N$, let $\sem{X}_v$, $\sem{Y}_v$ and $R = \sem{\Pi}_v$ as above. The assumption that the premises of $\omega\mand$ are all satisfied boil down to the inclusion $(R^{n})^{-1}[\sem{X}_v]\subseteq \sem{Y}_v$ holding for every $n\geq 1$. Hence, \[(R^+)^{-1}[\sem{X}_v] = \bigcup_{n\geq 1}(R^{n})^{-1}[\sem{X}_v]\subseteq \sem{Y}_v,\]
as required. The soundness of the remaining $\omega$-rules is shown similarly, and is omitted.

\section{Completeness}
\label{PDL:sec:Completeness}

In the present section, we discuss the  completeness of the Dynamic Calculus for PDL w.r.t.\ the semantics of Section \ref{PDL:sec:Soundness}. We show that the translation (cf.\ Section \ref{PDL:translation PDL to DC}) of each of the PDL axioms is derivable in the Dynamic Calculus. 
Unlike what we did in \cite{Multitype}, here we need to consider all possible version of the axioms arising from the disambiguation procedure.
Our completeness proof is indirect, and relies on the fact that PDL is complete   w.r.t.\ the standard Kripke semantics, and that the translation preserves the semantic interpretation on the standard models (as discussed in Section \ref{PDL:sec:Soundness}).

In the present section, we restrict our attention to deriving the box-versions of the  fix point  and induction axioms for PDL. The derivations of the remaining box-axioms for PDL are collected in Appendix \ref{Appendix : completeness for PDL}. The diamond-axioms can be also derived without appealing to the classical box/diamond interdefinability. These derivations follow a similar pattern to the ones given below and in Appendix \ref{Appendix : completeness for PDL}; the details are omitted.

\noindent
$\rule{146.2mm}{0.5pt}$

\textbf{Box-Fix point $+$}\ \ \ $(\alpha \mra A) \pand (\alpha \mra (\alpha^+ \mra A)) \,\dashv\vdash\, \alpha^+ \mra A$

\begin{center}
\begin{tabular}{@{}ll}

\AX$\alpha \fCenter \alpha$
\AX$A \fCenter A$
\BI$\alpha \mra A \fCenter \alpha \mRA A$
\UI$\alpha \mBAND \alpha \mra A \fCenter A$
%
\AX$\alpha \fCenter \alpha$

\AX$\alpha \fCenter \alpha$
\UI$\alpha^\oplus \fCenter \alpha^+$
\AX$A \fCenter A$
\BI$\alpha^+ \mra A \fCenter \alpha^\oplus \mRA A$
\BI$\alpha \mra (\alpha^+ \mra A) \fCenter \alpha \mRA (\alpha^\oplus \mRA A)$
\UI$\alpha \mBAND \alpha \mra (\alpha^+ \mra A) \fCenter \alpha^\oplus \mRA A$
\UI$\alpha^\oplus \mBAND (\alpha \mBAND \alpha \mra (\alpha^+ \mra A)) \fCenter A$
\UI$(\alpha \,{\textbf{\scriptsize{;}}}\, \alpha^\oplus) \mBAND \alpha \mra (\alpha^+ \mra A) \fCenter A$
\LeftLabel{\emph{FP $\blacktriangle$}}
\BI$\alpha^\oplus \mBAND (\alpha \mra A \,, \alpha \mra (\alpha^+ \mra A)) \fCenter A$
\UI$\alpha^\oplus \fCenter A \mLA (\alpha \mra A \,, \alpha \mra (\alpha^+ \mra A))$
\UI$\alpha^+ \fCenter A \mLA (\alpha \mra A \,, \alpha \mra (\alpha^+ \mra A))$
\UI$\alpha^+ \mBAND (\alpha \mra A \,, \alpha \mra (\alpha^+ \mra A)) \fCenter A$
\UI$\alpha \mra A \,, \alpha \mra (\alpha^+ \mra A) \fCenter \alpha^+ \mRA A$
\UI$\alpha \mra A \,, \alpha \mra (\alpha^+ \mra A) \fCenter \alpha^+ \mra A$
\UI$\alpha \mra A \pand \alpha \mra (\alpha^+ \mra A) \fCenter \alpha^+ \mra A$
\DisplayProof
\\

\end{tabular}
\end{center}

\begin{center}
\begin{tabular}{@{}ll}

\AX$\alpha \fCenter \alpha$
\UI$\alpha^\oplus \fCenter \alpha^+$
\AX$A \fCenter A$
\BI$\alpha^+ \mra A\fCenter \alpha^\oplus \mRA A$
\UI$\alpha^+ \mra A \fCenter \alpha \mRA A$
\UI$\alpha^+ \mra A \fCenter \alpha \mra A$
\AX$\alpha \fCenter \alpha$
\UI$\alpha^\oplus \fCenter \alpha^+$
\AX$\alpha \fCenter \alpha$
\UI$\alpha^\oplus \fCenter \alpha^+$
\UI$\alpha^+ \fCenter \alpha^+$
\LeftLabel{\emph{abs$_4$}}
\BI$\alpha^\oplus \,{\textbf{\scriptsize{;}}}\, \alpha^+ \fCenter \alpha^{+\ominus}$
\UI$ (\alpha^\oplus \,{\textbf{\scriptsize{;}}}\, \alpha^+)^{\oplus} \fCenter \alpha^+$


%
\AX$A \fCenter A$
\BI$\alpha^+ \mra A \fCenter (\alpha^\oplus \,{\textbf{\scriptsize{;}}}\, \alpha^+)^{\oplus} \mRA_{\!0\,} A$
\UI$\alpha^+ \mra A \fCenter (\alpha^\oplus \,{\textbf{\scriptsize{;}}}\, \alpha^+) \mRA_{\!1\,} A$
\UI$\alpha^+ \mra A \fCenter \alpha^\oplus \mRA (\alpha^+ \mRA_{\!0\,} A)$
\UI$\alpha^\oplus \mBAND_{\!0\,} \alpha^+ \mra A \fCenter \alpha^+ \mRA_{\!0\,} A$
\UI$\alpha^\oplus \mBAND_{\!0\,} \alpha^+ \mra A \fCenter \alpha^+ \mra A$
\UI$\alpha^+ \mra A \fCenter \alpha^\oplus \mRA (\alpha^+ \mra A)$
\UI$\alpha^+ \mra A \fCenter \alpha \mRA_1 (\alpha^+ \mra A)$
\UI$\alpha^+ \mra A \fCenter \alpha \mra (\alpha^+ \mra A)$
\BI$\alpha^+ \mra A\,, \alpha^+ \mra A \fCenter (\alpha \mra A) \pand (\alpha \mra (\alpha^+ \mra A))$
\LeftLabel{$C_L$}
\UI$\alpha^+ \mra A \fCenter (\alpha \mra A) \pand (\alpha \mra (\alpha^+ \mra A))$
\DisplayProof
 \\

\end{tabular}
\end{center}

\noindent
$\rule{146.2mm}{0.5pt}$

\textbf{Box-Induction $+$}\ \ \ $(\alpha \mra A) \pand (\alpha^+ \mra (A \pra (\alpha \mra A))) \,\vdash\, \alpha^+ \mra A$
\label{Appendix:PDL:Completeness:AxiomOmegaRule}
\smallskip

\noindent The following (incomplete) derivation takes us to the point in which the infinitary rule $\omega \vartriangleleft$ is applied:

\begin{center}
\AXC{$
\left(\begin{array}{l|c}
\ \ \ \ \ \ \ \pi_{n}                                                                                                                                   &               \\
\ \ \ \ \ \ \ \vdots                                                                                                                                   & n \geq 1 \\
\alpha^{(\mathbf{n})} \mBAND 
(\alpha \mra A \,, \alpha^+ \mra (A \pra (\alpha \mra A)))
\fCenter \,A \, &
\end{array}\right)
$}
\RightLabel{$\omega \blacktriangle$}
\UIC{$\alpha^\oplus \mBAND (\alpha \mra A \,, \alpha^+ \mra (A \pra (\alpha \mra A))) \fCenter A$}
\UIC{$\alpha^\oplus \fCenter A \mSLA (\alpha \mra A \,, \alpha^+ \mra (A \pra (\alpha \mra A)))$}
\UIC{$\alpha^+ \fCenter A \mSLA (\alpha \mra A \,, \alpha^+ \mra (A \pra (\alpha \mra A)))$}
\UIC{$\alpha^+ \mBAND (\alpha \mra A \,, \alpha^+ \mra (A \pra (\alpha \mra A))) \fCenter A$}
\UIC{$\alpha \mra A \,, \alpha^+ \mra (A \pra (\alpha \mra A)) \fCenter \alpha^+ \mRA A$}
\UIC{$(\alpha \mra A) \pand (\alpha^+ \mra (A \pra (\alpha \mra A))) \fCenter \alpha^+ \mRA A$}
\UIC{$(\alpha \mra A) \pand (\alpha^+ \mra (A \pra (\alpha \mra A))) \fCenter \alpha^+ \mra A$}
\RightLabel{notation}
\UIC{$\ls\alpha\rs A \pand \ls\alpha^+\rs (A \pra \ls\alpha\rs A) \fCenter \ls\alpha^+\rs A$}
\DisplayProof
\end{center}
\medskip


\noindent To complete the proof we are reduced to showing that each premise of the application of the $\omega \blacktriangle$ rule is derivable, that is:
\begin{proposition}
\label{PDL:prop:induction}
The following sequent is derivable for any $n \geqslant 1$: $$\alpha^{(n)} \mBAND 
(\alpha \mra A \,, \alpha^+ \mra (A \pra (\alpha \mra A)))
\fCenter \,A .$$
\end{proposition}
\noindent In what follows, the abbreviations below will be useful:
\begin{itemize}
\item let $\alpha^{\mathbf{{(\odot n)}}} (-)$ abbreviate $\underbrace{\alpha \odot ( \alpha \odot \ldots (\alpha \,\odot}_{n} \,(\,-\,)) \ldots )$, for $\odot \in \{\mAND, \mBAND, \mRA, \mBRA\}$;
\item let $\alpha^{(\cdot n)} (-)$ abbreviate $\underbrace{\alpha \cdot ( \alpha \cdot \ldots (\alpha \,\cdot}_{n} \,(\,-\,)) \ldots )$, for $\cdot \in \{\mand, \mband, \mra, \mbra\}$;
\item let $\alpha^{\mathbf{(n)}}$ and $\alpha^{n}$ abbreviate $\underbrace{\alpha \gSC ( \alpha \gSC \ldots (\alpha \,\gSC \alpha}_{n}) \ldots )$ and $\underbrace{\alpha \gsc ( \alpha \gsc \ldots (\alpha \,\gsc \alpha}_{n}) \ldots )$, respectively.
\end{itemize}

\begin{lemma}
\label{PDL:lemma:n,n>n+1}
Let $B = A \pra (\alpha\mra A)$. The following sequent is derivable for each $n \geqslant 1$:
$$\alpha^{(\mra n)} (A) \,, \alpha^{(\mra n)} (B) \fCenter \alpha^{(\mRA \mathbf{n+1})} (A).$$
\end{lemma}

\begin{proof}
The statement is proved by the following schematic derivation.

\begin{center}
{\small{
\AX$\alpha \fCenter \alpha$
\AX$\alpha \fCenter \alpha$
%
\AX$\alpha \fCenter \alpha$
\AX$A \fCenter A$
\BI$\alpha\mra A \fCenter \alpha\mRA A$
\LeftLabel{$n-1$ appl's of $\mra_L$}
\dashedLine
\BI$\alpha^{(\mra n)} (A) \fCenter \alpha^{(\mRA \mathbf{n})} A$
\LeftLabel{$n$ appl's of $\vartriangleright \blacktriangle$}
\dashedLine
\UI$\alpha^{(\mBAND \mathbf{n})} (\alpha^{(\mra n)} (A)) \fCenter A$

\AX$\alpha \fCenter \alpha$
\AX$A \fCenter A$
\BI$\alpha\mra A \fCenter \alpha\mRA A$
\BI$A \pra (\alpha\mra A) \fCenter \alpha^{(\mBAND \mathbf{n})} (\alpha^{(\mra n)} (A)) > (\alpha\mRA A)$
\LeftLabel{notation}
\UI$B \fCenter \alpha^{(\mBAND \mathbf{n})} (\alpha^{(\mra n)} (A)) > (\alpha\mRA A)$
\dashedLine
\LeftLabel{$n$ appl's of $\mra_L$}
\BI$\alpha^{(\mra n)} (B) \fCenter \alpha^{(\mRA \mathbf{n})} (\alpha^{(\mBAND \mathbf{n})} (\alpha^{(\mra n)} (A)) > (\alpha\mRA A))$
\dashedLine
\LeftLabel{$n$ appl's of $\vartriangleright \blacktriangle$}
\UI$\alpha^{(\mBAND \mathbf{n})} \alpha^{(\mra n)} (B) \fCenter \alpha^{(\mBAND \mathbf{n})} (\alpha^{(\mra n)} (A)) > (\alpha\mRA A)$
\UI$(\alpha^{(\mBAND \mathbf{n})} (\alpha^{(\mra n)} (A))) \,, (\alpha^{(\mBAND \mathbf{n})} \alpha^{(\mra n)} (B)) \fCenter \alpha\mRA A$
\dashedLine
\RightLabel{$n$ appl's of $\mband$ dis}
\UI$\alpha^{(\mBAND \mathbf{n})} (\alpha^{(\mra n)} (A) \,, \alpha^{(\mra n)} (B)) \fCenter \alpha\mRA A$

\dashedLine
\RightLabel{$n$ appl's of $\mband \vartriangleright$}

\UI$\alpha^{(\mra n)}  (A) \,, \alpha^{(\mra n)} (B) \fCenter \alpha^{(\mRA \mathbf{n+1})} (A)$
\DisplayProof
}}
\end{center}

\end{proof}

\begin{corollary}
\label{PDL:corollary:n,n>n+1}
Let $B = A \pra (\alpha\mra A)$. The following sequent is derivable for each $n \geqslant 1$:
$$\alpha^{(\mra n)} (A) \,, \alpha^{(\mra n)} (B) \fCenter \alpha^{(\mra n+1)} (A).$$
\end{corollary}

\begin{proof}
The schematic derivation in the proof of Lemma \ref{PDL:lemma:n,n>n+1} shows in particular that a derivation for the following sequent exists:
$$\alpha^{(\mBAND \mathbf{n})} (\alpha^{(\mra n)} (A) \,, \alpha^{(\mra n)} (B)) \fCenter \alpha\mRA A.$$
Then, the desired derivation can be obtained by prolonging that derivation with $n$ alternations of $\mra_R$ and $\mband \vartriangleright$, as follows:
\begin{center}
\AX$\alpha^{(\mBAND \mathbf{n})} (\alpha^{(\mra n)} (A) \,, \alpha^{(\mra n)} (B)) \fCenter \alpha\mRA A$
\UI$\alpha^{(\mBAND \mathbf{n})} (\alpha^{(\mra n)} (A) \,, \alpha^{(\mra n)} (B)) \fCenter \alpha\mra A$
\UI$\alpha^{(\mBAND \mathbf{n-1})} (\alpha^{(\mra n)} (A) \,, \alpha^{(\mra n)} (B)) \fCenter \alpha\mRA(\alpha\mra A)$
\UI$\alpha^{(\mBAND \mathbf{n-1})} (\alpha^{(\mra n)} (A) \,, \alpha^{(\mra n)} (B)) \fCenter \alpha\mra(\alpha\mra A)$
\dashedLine
\UI$\alpha^{(\mra n)} (A) \,, \alpha^{(\mra n)} (B) \fCenter \alpha^{(\mra n+1)} (A)$
\DisplayProof
\end{center}
\end{proof}


\begin{lemma}
\label{PDL:lemma:1,1...n-1,n>n+1}
Let  $B = A \pra (\alpha\mra A)$. The following sequent is derivable for each $n \geqslant 1$:
$$\alpha^{(\mra 1)} (A) \,, \alpha^{(\mra 1)} (B) \,, \ldots \,, \alpha^{(\mra n-1)} (B) \,, \alpha^{(\mra n)} (B) \fCenter \alpha^{(\mRA \mathbf{n+1})} (A).$$
\end{lemma}
\begin{proof}
Fix $n\geqslant 1$, let $X_{n+1}$ abbreviate $\alpha^{(\mRA \mathbf{n+1})} (A)$, and for each $1\leqslant i< n$, let $C_i$ and $D_i$ abbreviate $\alpha^{(\mra i)} (A)$ and $\alpha^{(\mra i)} (B)$, respectively. By Corollary \ref{PDL:corollary:n,n>n+1}, a derivation $\pi_i$ of $C_i \,, D_i \fCenter C_{i+1}$ is available for each $1\leqslant i< n$, and by Lemma \ref{PDL:lemma:n,n>n+1}, a derivation $\pi_n$ of $C_n \,, D_n \fCenter X_{n+1}$ is also available. Then the following derivation, which essentially consists in $n-1$ applications of Cut, proves the statement:
\begin{center}
{\small{
\AXC{$\pi_1$}
\noLine
\UIC{$\vdots$}
\noLine
\UIC{$C_1\,, D_1 \fCenter C_2$}
\AXC{$\pi_2$}
\noLine
\UIC{$\vdots$}
\noLine
\UIC{$C_2 \,, D_2 \fCenter C_3$}
\UIC{$C_2 \fCenter C_3 < D_2$}
\RightLabel{\emph{Cut$_1$}}
\BIC{$C_1 \,, D_1 \fCenter C_3 < D_2$}
\UIC{$C_1 \,, D_1 \,, D_2 \fCenter C_3$}
\AXC{$\pi_3$}
\noLine
\UIC{$\vdots$}
\noLine
\UIC{$C_3 \,, D_3 \fCenter C_4$}
\UIC{$C_3 \fCenter C_4 < D_3$}
\RightLabel{\emph{Cut$_2$}}
\BIC{$C_1 \,, D_1 \,, D_2 \fCenter C_4 < D_3$}
\UIC{$C_1 \,, D_1 \,, D_2 \,, D_3 \fCenter C_4$}

\AXC{$\pi_i$}
\noLine
\UIC{$\vdots$}
\noLine
\UIC{\phantom{A}}
\noLine
\UIC{$\cdots$}
\AXC{$\pi_n$}
\noLine
\UIC{$\vdots$}
\noLine
\UIC{$C_n \,, D_n \fCenter X_{n+1}$}
\UIC{$C_n \fCenter X_{n+1} < D_n$}
\RightLabel{\emph{Cut$_{n-1}$}}
\dashedLine
\TIC{$C_1 \,, D_1\,, \ldots \,, D_{n-1} \fCenter X_{n+1} < D_n$}
\UIC{$C_1 \,, D_1\,, \ldots \,, D_{n-1}\,, D_n \fCenter X_{n+1}$}
\UIC{$\alpha^{(\mra 1)} (A) \,, \alpha^{(\mra 1)} (B) \,, \ldots \,, \alpha^{(\mra n-1)} (B) \,, \alpha^{(\mra n)} (B) \fCenter \alpha^{(\mRA n+1)} (A)$}
\DisplayProof
}}
\end{center}
\end{proof}


\begin{lemma}
\label{PDL:lemma:+>n}
The following sequent is derivable for each $n \geqslant 1$ and  every formula $C$: $$\alpha^+ \mra C \,\vdash\, \alpha^{(\mra n)} (C).$$
\end{lemma}
\begin{proof}
For $n=1$, the following derivation proves the statement:

\begin{center}
\AX$\alpha \fCenter \alpha$
\UI$\alpha^\oplus \fCenter \alpha^+$
\AX$C \fCenter C$
\BI$\alpha^+ \mra C \fCenter \alpha^\oplus \mRA C$
\UI$\alpha^+ \mra C \fCenter \alpha \mRA C$
\UI$\alpha^+ \mra C \fCenter \alpha \mra C$
\DisplayProof
\end{center}

\noindent For $n \geqslant 2$, the following schematic derivation proves the statement:

\begin{center}
\AX$\alpha \fCenter \alpha$
\UI$\alpha^\oplus \fCenter \alpha^+$
\UI$\alpha \fCenter \alpha^{+ \ominus}$

\AX$\alpha \fCenter \alpha$
\UI$\alpha^\oplus \fCenter \alpha^+$
\UI$\alpha \fCenter \alpha^{+ \ominus}$

\def\fCenter{\cdots}
\AX$\fCenter$

\def\fCenter{\ \vdash \ }
\AX$\alpha \fCenter \alpha$
\UI$\alpha^\oplus \fCenter \alpha^+$
\UI$\alpha \fCenter \alpha^{+ \ominus}$
\AX$\alpha \fCenter \alpha$
\UI$\alpha^\oplus \fCenter \alpha^+$
\UI$\alpha \fCenter \alpha^{+ \ominus}$
\RightLabel{\emph{abs}}
\BI$\alpha \,\textbf{;}\, \alpha \fCenter \alpha^{+ \ominus}$

\dashedLine
\RightLabel{$n-2$ appl's of \emph{abs}}
\TI$\alpha^{\textbf{(n-1)}} \fCenter \alpha^{+ \ominus}$

\BI$\alpha \,\textbf{;}\, \alpha^{\textbf{(n-1)}} \fCenter \alpha^{+ \ominus}$
\UI$(\alpha \,\textbf{;}\, \alpha^{\textbf{(n-1)}})^\oplus \fCenter \alpha^+$

\AX$C \fCenter C$

\BI$\alpha^+ \mra C \fCenter (\alpha \,\textbf{;}\, \alpha^{\textbf{(n-1)}})^\oplus \mRA C$
\UI$\alpha^+ \mra C \fCenter (\alpha \,\textbf{;}\, \alpha^{\textbf{(n-1)}}) \mRA C$
\RightLabel{${act} \vartriangleright$}
\UI$\alpha^+ \mra C \fCenter \alpha \mRA (\alpha^{\textbf{(n-1)}} \mRA C)$
\RightLabel{$\vartriangleright \blacktriangle$}
\UI$\alpha \mBAND \alpha^+ \mra C \fCenter \alpha^{\textbf{(n-1)}} \mRA C$
\RightLabel{notation}
\UI$\alpha \mBAND \alpha^+ \mra C \fCenter (\alpha \,\textbf{;}\, \alpha^{\textbf{(n-2)}}) \mRA C$
\dashedLine
\RightLabel{$(\ast)$}

\UI$\alpha \mBAND (\alpha^{\textbf{($\mBAND$ n-2)}} (\alpha^+ \mra C)) \fCenter \alpha \mRA C$
\RightLabel{$\mra$}
\UI$\alpha \mBAND (\alpha^{\textbf{($\mBAND$ n-2)}} (\alpha^+ \mra C)) \fCenter \alpha \mra C$
\RightLabel{$\blacktriangle \vartriangleright$}
\UI$\alpha^{\textbf{($\mBAND$ n-2)}} (\alpha^+ \mra C) \fCenter \alpha \mRA \alpha \mra C$
\dashedLine
\RightLabel{$(\ast\ast)$}
\UI$\alpha^+ \mra C \fCenter \alpha^{(\mra n)} (C)$
\DisplayProof
\end{center}

$(\ast)$ $n-2$ alternating applications of the structural rule   {\em act}$\vartriangleright$ and of the display postulate for the connectives $\vartriangleright$ and $ \blacktriangle$.

$(\ast\ast)$ $n-1$ alternating applications of the operational rule $\mra_R$ and of the display postulate for the connectives $\blacktriangle$ and $ \vartriangleright$.
\end{proof}


\begin{lemma}
\label{PDL:lemma:+(1,1...n-1,n>n+1)}
Let $B = A \pra (\alpha\mra A)$. The following sequent is derivable for each $n \geqslant 1$:
$$\alpha^{(\mra 1)}(A) \,, \alpha^+ \mra B \fCenter \alpha^{(\mRA \mathbf{n})}(A).$$
\end{lemma}
\begin{proof}
Fix $n\geqslant 1$, let $X_{n+1}$ abbreviate $\alpha^{(\mRA \mathbf{n+1})} (A)$, let $D^+$ abbreviate $\alpha^+\mra B$, and for each $1\leqslant i\leqslant n$, let $C_i$ and $D_i$ abbreviate $\alpha^{(\mra i)} (A)$ and $\alpha^{(\mra i)} (B)$, respectively.
By Lemma \ref{PDL:lemma:1,1...n-1,n>n+1}, a derivation $\pi_n$ of the sequent $$C_1 \,, D_1 \,, \ldots \,, D_n \fCenter X_{n+1}$$ is available for each $n \geqslant 1$ and for $B = A \pra (\alpha \mra A)$.

By Lemma \ref{PDL:lemma:+>n}, for each $1\leqslant i\leqslant n$, a derivation $\pi'_i$ of the sequent $\alpha^+ \mra C \,\vdash\, \alpha^{(\mra i)}(C)$ is available for any $C$, so in particular for $C = B$ we get a derivation of $$D^+\vdash D_i.$$

Applying Cut $n-1$ times, the following derivation proves the statement:
\begin{center}
\AXC{$\pi'_n$}
\noLine
\UIC{$\vdots$}
\noLine
\UIC{$D^+ \fCenter D_n$}
\AXC{$\pi'_i$}
\noLine
\UIC{$\vdots$}
\noLine
\UIC{\phantom{A}}
\noLine
\UIC{$\cdots$}

\AXC{$\pi'_1$}
\noLine
\UIC{$\vdots$}
\noLine
\UIC{$D^+ \fCenter D_1$}
\AXC{$\pi_n$}
\noLine
\UIC{$\vdots$}
\noLine
\UIC{$C_1 \,, D_1 \,, \ldots \,, D_n \fCenter X_{n+1}$}
\RightLabel{$D_1$ in display}
\dashedLine
\UIC{$D_1 \fCenter Y_1$}
\RightLabel{\emph{Cut$_1$}}
\BIC{$D^+ \fCenter Y_1$}
\RightLabel{$D_2$ in display}
\dashedLine
\UIC{$D_2 \fCenter Y_2$}
\RightLabel{$n-2$ sequences of display and Cut}
\dashedLine
\UIC{$D_n \fCenter Y_n$}
\RightLabel{Cut$_{n-1}$}
\TIC{$D^+\fCenter Y_n$}
\RightLabel{Display}
\dashedLine
\UIC{$C_1\,, \underbrace{D^+ \,, \ldots \,, D^+}_{n} \fCenter X_{n+1}$}
\RightLabel{Contraction}
\dashedLine
\UIC{$C_1 \,, D^+\fCenter X_{n+1}$}


\DisplayProof
\end{center}
\end{proof}
\noindent Now we can finish the proof of Proposition \ref{PDL:prop:induction} as follows:
\begin{proof}
By Lemma \ref{PDL:lemma:+(1,1...n-1,n>n+1)},  a derivation of the sequent $\alpha^{(\mra 1)} (A) \,, \alpha^+ \mra B \fCenter \alpha^{(\mRA \mathbf{n})} (A)$ exists; then the desired derivation is obtained by prolonging that derivation as shown below.

\begin{center}
\AX$\alpha^{(\mra 1)} (A) \,, \alpha^+ \mra B \fCenter \alpha^{(\mRA \mathbf{n})} (A)$
\dashedLine
\RightLabel{$n$ appl's of $\vartriangleright \blacktriangle$}
\UI$\alpha^{(\mBAND \mathbf{n})} (\alpha^{(\mra 1)}(A) \,, \alpha^+ \mra B) \fCenter A$
\dashedLine
\RightLabel{$n$ appl's of ${act} \blacktriangle$}
\UI$\alpha^{(\mathbf{n})} \mBAND (\alpha^{(\mra 1)}(A) \,, \alpha^+ \mra B) \fCenter A$
\DisplayProof
\end{center}
\end{proof}


\section{Cut-elimination}
\label{PDL:sec:Cut-elimination}

In the present section, we prove that the multi-type display calculus for PDL is a proper display calculus (cf.\ Definition \ref{PDL:def:quasi-displayable multi-type calculus}). By Theorem \ref{PDL:thm:meta multi}, this is enough to establish that the calculus enjoys the cut elimination and the subformula property.
Conditions C$_1$, C$_2$, C$_3$, C$_4$, C$_5$, C''$_5$, C'$_6$, C'$_7$ and C$_{10}$ are straightforwardly verified by inspecting the rules and are left to the reader. Condition  C'$_2$ can be straightforwardly verified by inspection on the rules, for instance by observing that the domains and codomains of adjoints are rigidly determined.

The following proposition shows that condition C$_9$ is met:
\begin{proposition}
\label{PDL:prop:type regular}
Any derivable sequent in the calculus for PDL is type-uniform.
\end{proposition}

\begin{proof}
We prove the  proposition by induction on the height of the derivation.
The base case is verified by inspection; indeed,  the  following axioms are type-uniform by definition of their constituents:

\[ \pi \vdash \pi \quad p \vdash p \quad \bot \vdash \textrm{I}  \quad \textrm{I} \vdash \top \]

\noindent As to the inductive step, one can verify by inspection that all the rules of the calculus preserve type-uniformity, and that the Cut rules are strongly type-uniform.
\end{proof}
\noindent Finally, the verification steps for C'$_8$ are collected in Appendix \ref{Appendix : Cut elimination for PDL, principal stage}.

\section{The open issue of conservativity}
\label{PDL:sec:Conservativity}

In the present section, we expand on the difficulties encountered in the proof of conservativity for the Dynamic Calculus for PDL. 

\paragraph*{Semantic argument.} 
The main avenue to prove the conservativity of a display calculus w.r.t.\ the original logic that the calculus is meant to capture is semantic. 
Namely, if the original logic is complete w.r.t.\ a given semantics, then it is enough to prove that every rule is sound w.r.t.\ that semantics.
This is not possible in the case of the Dynamic Calculus for PDL, since some display rules are not interpretable in the semantics due to the presence of virtual adjoints (cf.\ Section \ref{PDL:sec:language and rules}).
This situation is analogous to that of the Dynamic Calculus for EAK.

\paragraph*{Syntactic elimination of virtual adjoints.}
In the  setting of the Dynamic Calculus for EAK, our proof was syntactical, and its pivot step was showing that any valid proof-tree the root of which is operational and  of type $\mathsf{Fm}$ can be rewritten into a valid proof-tree involving no virtual adjoints (cf.\
\cite[Section 7]{Multitype}).
This process of removing virtual adjoints could take place essentially because the action-grammar of EAK was very poor.
In the case of PDL, because the presence of the iteration in the
grammar of actions, this fact is not true. 
However, the fact that virtual adjoints occur in essential ways in derivation trees of   operational sequents $A\vdash B$  of type $\mathsf{Fm}$
does not imply per se that the calculus is not sound w.r.t.\ the original language.

\paragraph*{Display-Type Calculi.} 
Another option would be modifying the Dynamic Calculus for PDL so as to make it a display-type calculus rather than a display calculus. 
The modifications would require removing all the rules involving virtual adjoints and replacing the cut rules with suitable surgical cuts.
However, in order to obtain a complete calculus, certain rules which are derived in the original Dynamic Calculus would also need to be added.
This is the case of the following rule:
\[
\AX $ \alpha^{\oplus} \mBAND X \fCenter Y $
\UI $ \alpha^{+} \mBAND X \fCenter Y $
\DisplayProof
\]
Unfortunately, the rule above violates condition C$_5$ requiring that principal constituents be displayed. Hence, 
cut elimination cannot be proved for the resulting calculus via Theorem 4.1 in \cite{TrendsXIII}.

\paragraph*{Conservativity via translation.}
In \cite{DCGT14} and \cite{CDGT13ext}, the conservativity issue for a display calculus for Full Intuitionistic Linear Logic (FILL) was resolved with a technique which we intend to adapt to PDL.
This adaptation is still work in progress.
In what follows, we report on the proof strategy adopted in \cite{DCGT14}, and discuss its possible adaptations.
The main steps in the proof strategy are:
\begin{enumerate}
\item  define a sound and complete display calculus for
an extension of the logic with additional adjunctions.
The extension considered in \cite{DCGT14} is  Bi-Intuitionistic Linear Logic (BiILL).

\item  translate the display calculus to a shallow inference nested sequent calculus.

\item  translate the shallow inference nested sequent  calculus to a deep inference nested sequent calculus.

\item  prove that the deep inference nested sequent calculus is sound with respect to the original logic. In the case of \cite{CFPS14}, the authors prove that the deep inference nested sequent calculus is sound with respect to FILL.
\end{enumerate}
The adaptation of this technique to  the  setting of PDL is not straightforward.
For instance, the first translation  transforms logical connectives into meta-linguistic data structures such as $\vdash$ without losing information. 
The naive adaptation of  this step to the setting of PDL
would make us lose information. This direction is still work in progress.

\section{Conclusions}
\label{PDL:sec:Conclusions}
   The calculus introduced in the present chapter is an attempt at extending the methodology of display calculi to a fully-fledged PDL-type setting. Previous attempts in this direction (e.g.\ \cite{Wa98}) exclude both the Kleene star and the positive iteration. Accounting for these operations is proof-theoretically challenging, and indeed, the existing proposals in the literature, also outside the display calculus methodology, typically witness a trade-off between achieving \emph{syntactical} full cut elimination at the price of having infinitary rules in the system (e.g.\ \cite{Poggiolesi}), or dispensing with infinitary rules at the price of achieving cut elimination modulo analytic cut(s) (e.g.\ unpublished manuscript \cite{Har?}). The present proposal aims at paving the way for  escaping this trade-off. Indeed, our starting point is the basic understanding that the induction axioms/induction rules (which are the main hurdle to a smooth proof-theoretic treatment of PDL)  ingeniously encode {\em by means of formulas} a piece of information which by rights pertains to {\em actions}; namely, they encode the relation between an action and its (reflexive and) transitive closure. This encoding is done either by resorting to infinitary axioms/rules, or by introducing some forms of `loops' (i.e.\ formulas appearing both in the antecedent and in the consequent of an implication). Each of these two ways gives rise to issues which hinder a smooth proof-theoretic treatment of PDL. Taken together, these two alternatives are at the basis of the trade-off we wish to escape. Our idea for a solution (which needs to be perfected) involves introducing enough expressivity in the language so that formulas are not to be relied upon  anymore  to encode a piece of information which strictly speaking pertains purely to {\em actions}, and neither pertains to formulas, nor to the interaction between formulas and actions. 

In particular,  we aim at describing the proof-theoretic behaviour  of the positive  iteration operation $+$ in terms of the order-theoretic behaviour of the transitive closure. Namely, we make use of the well known fact that the map associating each binary relation on a given set $W$ to its transitive closure  can be characterized order-theoretically as the left adjoint of the inclusion map sending the transitive relations on $W$ into $\mathcal{P}(W\times W)$. The introduction of {\em two} different types of actions is then motivated by the need  to  properly express this adjunction.  Thus, we expand the language, both at the structural and at the operational level, with the following pair of adjoint maps:
\[
(\cdot)^+ :  \mathsf{Act} \to \mathsf{TAct}\quad\quad (\cdot)^- : \mathsf{TAct} \to \mathsf{Act}.
\]
The adjunction relation $(\cdot)^+\dashv (\cdot)^-$ is not enough to capture the informational content of the transitive closure. The missing pieces are: (1) the map $(\cdot)^-$ being an order-embedding; (2) the fact that the $\mathsf{TAct}$-type elements are transitive, i.e.\ $\delta; \delta \vdash \delta$ for each $\delta \in \mathsf{TAct}$. Neither piece of information is captured at the operational level. Indeed, we can only prove 

\begin{center}
\AX$\delta \fCenter \delta$
\UI$\delta^- \fCenter \delta^\ominus$
\UI$\delta^{-\oplus} \fCenter \delta$
\UI$\delta^{-+} \fCenter \delta$
\DisplayProof
%
%
\\
\end{center} 
\noindent 
Hence, we had to resort to the omega-induction rules (which, besides being infinitary, take the form of interaction rules between formula-type and action-type terms) to encode the transitive closure and derive the induction axioms. We conjecture that being able to express transitive closure at the structural level is key to dispensing with the infinitary rules, which is our next goal for future developments in this line research. 

Finally, it is perhaps worth stressing that considerations such as the ones just made above can be made in a meaningful way only in the context of a multi-type environment in which actions and formulas enjoy equal standing as first-class citizens. Thus, the multi-type approach can also function as a `conceptual  tool', by means of which technical difficulties such as the ones mentioned above can be explained in terms of problems of {\em expressivity}. In their turn, properties and considerations involving different degrees of expressivity can  then be sharpened and made precise. 

\newpage

\section{Appendix}
\label{PDL:sec:Appendix}

\subsection{The Calculus for the Propositional Base of PDL}
\label{Appendix : The calculus for the propositional base of PDL}

\begin{center}
\begin{tabular}{rlcrl}
\mc{5}{c}{\textbf{Propositions Structural Rules}} \\

& & & & \\

\AX$X \fCenter Y$
\doubleLine
\LeftLabel{\fns$\textrm{I}^1_{L}$}
\UI$\textrm{I} \fCenter Y < X$
\DisplayProof
 &
\AX$X \fCenter Y$
\doubleLine
\RightLabel{\fns$\textrm{I}^1_{R}$}
\UI$X < Y \fCenter \textrm{I}$
\DisplayProof
 & &
\AX$X \fCenter Y$
\LeftLabel{\fns$\textrm{I}^2_{L}$}
\doubleLine
\UI$\textrm{I} \fCenter X > Y$
\DisplayProof
 &
\AX$X \fCenter Y$
\RightLabel{\fns$\textrm{I}^2_{R}$}
\doubleLine
\UI$Y > X \fCenter \textrm{I}$
\DisplayProof
 \\

  & & & & \\

\AX$X \fCenter Z$
\LeftLabel{\fns$W^1_{L}$}
\UI$Y \fCenter Z < X$
\DisplayProof
 &
\AX$X \fCenter Z$
\RightLabel{\fns$W^1_{R}$}
\UI$X < Z \fCenter Y$
\DisplayProof
 & &
\AX$X \fCenter Z$
\LeftLabel{\fns$W^2_{L}$}
\UI$Y \fCenter X > Z$
\DisplayProof
 &
\AX$X \fCenter Z$
\RightLabel{\fns$W^2_{R}$}
\UI$Z > X \fCenter Y$
\DisplayProof
 \\

 & & & & \\

\AX$X \,, X \fCenter Y$
\LeftLabel{\fns$C_L$}
\UI$X \fCenter Y $
\DisplayProof
 &
\AX$Y \fCenter X \,, X$
\RightLabel{\fns$C_R$}
\UI$Y \fCenter X$
\DisplayProof
 & &
\AX$X \,, (Y \,, Z) \fCenter W$
\LeftLabel{\fns$A_{L}$}
\UI$(X \,, Y) \,, Z \fCenter W $
\DisplayProof
&
\AX$W \fCenter (Z \,, Y) \,, X$
\RightLabel{\fns$A_{R}$}
\UI$W \fCenter Z \,, (Y \,, X)$
\DisplayProof
 \\

 & & & & \\

\AX$Y \,, X \fCenter Z$
\LeftLabel{\fns$E_L$}
\UI$X \,, Y \fCenter Z $
\DisplayProof
 &
\AX$Z \fCenter X \,, Y$
\RightLabel{\fns$E_R$}
\UI$Z \fCenter Y \,, X$
\DisplayProof
 & &
\AX$X > (Y \,; Z) \fCenter W$
\doubleLine
\LeftLabel{\fns\emph{Gri$_{L}$}}
\UI$(X > Y) \,; Z \fCenter W$
\DisplayProof
&
\AX$W \fCenter X > (Y \,; Z)$
\doubleLine
\RightLabel{\fns\emph{Gri$_{R}$}}
\UI$W \fCenter (X > Y) \,; Z$
\DisplayProof
 \\

 & & & & \\
\end{tabular}
\end{center}

\noindent The last rules in the table above, \emph{Gri}$_L$ and \emph{Gri}$_R$, are known as Grishin's rules: here they are useful to force the classical behaviour of our propositional base (if we remove \emph{Gri}, we will obtain a weaker logic cfr.\ \cite{Gore1}). 

\begin{center}
\begin{tabular}{rl}
\mc{2}{c}{\textbf{Propositions Display Postulates}} \\
 & \\

\AX$X \,, Y \fCenter Z$
\LeftLabel{\fns{$$}}
\doubleLine
\UI$Y \fCenter X > Z$
\DisplayProof
 &
\AX$Z \fCenter X \,, Y $
\RightLabel{\fns{$$}}
\doubleLine
\UI$X > Z \fCenter Y$
\DisplayProof
 \\

 & \\

\AX$X \,, Y \fCenter Z$
\LeftLabel{\fns{$$}}
\doubleLine
\UI$X \fCenter Z < Y$
\DisplayProof
 &
\AX$Z \fCenter X \,, Y $
\RightLabel{\fns{$$}}
\doubleLine
\UI$Z < Y \fCenter X$
\DisplayProof
 \\

 & \\
\end{tabular}
\end{center}

\noindent Below we list the rules for the operational connective (note that the latest three connectives with the name of the rule in brackets are those which do not belong to the language of the axioms that we have implicitly chosen).

\begin{center}
\begin{tabular}{rlrl}
\mc{4}{c}{\textbf{Propositions Operational Rules}} \\
 & & & \\

 \AXC{\phantom{$\gbot \fCenter \textrm{I}$}}
\LeftLabel{\fns$\bot_L$}
\UI$\bot \fCenter \textrm{I}$
\DisplayProof
 &
\AX$X \fCenter \textrm{I}$
\RightLabel{\fns$\bot_R$}
\UI$X \fCenter \bot$
\DisplayProof
 &
\AX$\textrm{I} \fCenter X$
\LeftLabel{\fns$\top_L$}
\UI$\top \fCenter X$
\DisplayProof
 &
\AXC{\phantom{$\textrm{I} \fCenter \top$}}
\RightLabel{\fns$\top_R$}
\UI$\textrm{I} \fCenter \top$
\DisplayProof
 \\

 & & & \\

\AX$A \,, B \fCenter X$
\LeftLabel{\fns$\pand_L$}
\UI$A \pand B \fCenter X$
\DisplayProof
 &
\AX$X \fCenter A$
\AX$Y \fCenter B$
\RightLabel{\fns$\pand_R$}
\BI$X \,, Y \fCenter A \pand B$
\DisplayProof
&
\AX$A \fCenter X$
\AX$B \fCenter Y$
\LeftLabel{\fns$\por_L$}
\BI$A \por B \fCenter X \,, Y$
\DisplayProof
 &
 \AX$X \fCenter A \,, B$
\RightLabel{\fns$\por_R$}
\UI$X \fCenter A \por B$
\DisplayProof
 \\

 & \\

\mc{2}{r}{
\AX$X \fCenter A$
\AX$B \fCenter Y$
\LeftLabel{\fns$\pra_L$}
\BI$A \pra B \fCenter X > Y$
\DisplayProof}
 &
\mc{2}{l}{
\AX$X \fCenter A > B$
\RightLabel{\fns$\pra_R$}
\UI$X \fCenter A \pra B$
\DisplayProof}
\\
\end{tabular}
\end{center}
\smallskip

\begin{center}
\begin{tabular}{rl}
\AX$B \fCenter Y$
\AX$X \fCenter A$
\LeftLabel{\fns$(\pla_L)$}
\BI$B \pla A \fCenter Y < X$
\DisplayProof
 &
\AX$Z \fCenter B < A$
\RightLabel{\fns$(\pla_R)$}
\UI$Z \fCenter B \pla A$
\DisplayProof
\\

 & \\

\AX$A > B \fCenter Z$
\LeftLabel{\fns$(\pdra_L)$}
\UI$A \pdra B \fCenter Z$
\DisplayProof
 &
\AX$A \fCenter X$
\AX$Y \fCenter B$
\RightLabel{\fns$(\pdra_R)$}
\BI$X > Y \fCenter A \pdra B$
\DisplayProof
 \\

 & \\

\AX$B < A \fCenter X$
\LeftLabel{\fns$(\pdla_L)$}
\UI$B \pdla A \fCenter X$
\DisplayProof
 &
\AX$Y \fCenter B$
\AX$A \fCenter X$
\RightLabel{\fns$(\pdla_R$)}
\BI$Y < X \fCenter B \pdla A$
\DisplayProof
 \\
\end{tabular}
\end{center}

{}

\subsection{Cut Elimination for PDL, Principal Stage}
\label{Appendix : Cut elimination for PDL, principal stage}

In the present subsection, we report on the verification of condition C'$_8$ of  the definition of quasi-proper  multi-type display calculi (cf.\ Section  \ref{PDL:ssec:quasi-def}).

Let us recall that C'$_8$ only concerns applications of the cut rules in which both occurrences of the given cut-term are {\em non parametric}.
Notice that non parametric occurrences of atomic terms of type $\mathsf{Fm}$ involve an axiom on at least  one premise, thus we are reduced to the following cases (the case of the constant $\bot$ is symmetric to the case of $\top$ and is omitted):
\begin{center}
\begin{tabular}{@{}lcrclcr@{}}
\bottomAlignProof
\AX$p \fCenter p$
\AX$p \fCenter p$
\BI$p \fCenter p$
\DisplayProof
 & $\rightsquigarrow$ &
\bottomAlignProof
\AX$p \fCenter p$
\DisplayProof
\ \ \ & \quad &\ \ \
\bottomAlignProof
\AX$\textrm{I} \fCenter \top$
\AXC{\ \ $\vdots$ \raisebox{1mm}{$\pi$}}
\noLine
\UI$\textrm{I} \fCenter X$
\UI$\top \fCenter X$
\BI$\textrm{I} \fCenter X$
\DisplayProof
 & $\rightsquigarrow$ &
\bottomAlignProof
\AXC{\ \ $\vdots$ \raisebox{1mm}{$\pi$}}
\noLine
\UI$\textrm{I} \fCenter X$
\DisplayProof
 \\
\end{tabular}
\end{center}
\noindent Notice that non parametric occurrences of any given (atomic) operational term $\pi$ of type $\mathsf{Act}$ are confined to axioms $\pi \vdash \pi$, so the proofs are analogous to the previous case of operational term $p$ of type $\mathsf{Fm}$ and they are omitted. In each of these cases, the cut in the original derivation is strongly-uniform by assumption, and is eliminated by the transformation.
As to cuts on non atomic terms, let us now restrict our attention to those cut-terms the main connective of which is $\mand_i, \mband_i, \mra_i, \mbra_i$  for $0\leq i \leq 1$. Here below we show the proofs only for the white heterogeneous connectives: the proofs for the black heterogeneous connectives are exactly the same modulo a uniform substitution of each white connective by the same black connective (both at the operational and structural level). 

\begin{center}
\begin{tabular}{@{}rcl@{}}
\bottomAlignProof
\AXC{\ \ \ $\vdots$ \raisebox{1mm}{$\pi_0$}}
\noLine
\UI$x \fCenter a$
\AXC{\ \ \ $\vdots$ \raisebox{1mm}{$\pi_1$}}
\noLine
\UI$y\ \fCenter\ b$
\BI$x \mAND_i y\ \fCenter\  a \mand_i b$
\AXC{\ \ \ $\vdots$ \raisebox{1mm}{$\pi_2$}}
\noLine
\UI$ a\mAND_i b\ \fCenter\ z$
\UI$ a \mand_i b\ \fCenter\ z$
\BI$x\mAND_i y\ \fCenter\ z$
\DisplayProof
 & $\rightsquigarrow$ &
\bottomAlignProof
\AXC{\ \ \ $\vdots$ \raisebox{1mm}{$\pi_1$}}
\noLine
\UI$y\ \fCenter\ b$
\AXC{\ \ \ $\vdots$ \raisebox{1mm}{$\pi_0$}}
\noLine
\UI$x \fCenter a$

\AXC{\ \ \,$\vdots$ \raisebox{1mm}{$\pi_2$}}
\noLine
\UI$ a \mAND_i b\ \fCenter\ z$
\UI$ a \ \fCenter\ z \mSBLA_i b$
\BI$ x  \fCenter\ z \mSBLA_i b$

\UI$ x \mAND_i b\ \fCenter\ z$
\UI$b\ \fCenter\ x\mBRA_i z$
\BI$y\ \fCenter\ x \mBRA_i z$
\UI$x \mAND_i y\ \fCenter\ z$
\DisplayProof
 \\
\end{tabular}
\end{center}

\begin{center}
\begin{tabular}{@{}rcl@{}}
\bottomAlignProof
\AXC{\ \ \ $\vdots$ \raisebox{1mm}{$\pi_1$}}
\noLine
\UI$y\ \fCenter\ a \mRA_{\!\!i}  b$
\UI$y\ \fCenter\ a \mra_{\!\!i}  b$

\AXC{\ \ \ $\vdots$ \raisebox{1mm}{$\pi_0$}}
\noLine
\UI$x \fCenter a$
\AXC{\ \ \ $\vdots$ \raisebox{1mm}{$\pi_2$}}
\noLine
\UI$b\ \fCenter\ z$
\BI$a \mra_{\!\!i}  b\ \fCenter\ x \mRA_{\!\!i} z$
\BI$y \fCenter x \mRA_{\!\!i}  z$
\DisplayProof
 & $\rightsquigarrow$ &
\bottomAlignProof
\AXC{\ \ \ $\vdots$ \raisebox{1mm}{$\pi_0$}}
\noLine
\UI$x \fCenter a$
\AXC{\ \ \ $\vdots$ \raisebox{1mm}{$\pi_1$}}
\noLine
\UI$y\ \fCenter\ a \mRA_{\!\!i}  b$
\UI$a \mBAND_i y \fCenter b$
\UI$a  \fCenter b \mSLA_i y$
\BI$x  \fCenter b \mSLA_i y$
\UI$x  \mBAND_i y \fCenter b$

\AXC{\ \ \ $\vdots$ \raisebox{1mm}{$\pi_2$}}
\noLine
\UI$b\ \fCenter\ z$
\BI$x \mBAND_i y\ \fCenter\ z$
\UI$y\ \fCenter\ x \mRA_{\!\!i}z$
\DisplayProof
 \\
\end{tabular}
\end{center}

\noindent In each  of these cases, the cut in the original derivation is strongly-uniform by assumption, and after the transformation, cuts of lower complexity are introduced which can be easily verified to be strongly-uniform for each $0\leq i\leq 1$.

Finally, let us consider the unary modalities test $\wn_i$ for $0\leq i\leq 1$, positive iteration $+$ and its left adjont $-$.
\begin{center}
\begin{tabular}{@{}rcl@{}}
\bottomAlignProof
\AXC{\ \ \ $\vdots$ \raisebox{1mm}{$\pi_1$}}
\noLine
\UI$X\ \fCenter\ A$
\UI$X\WN_i  \fCenter A \wn_i $
\AXC{\ \ \ $\vdots$ \raisebox{1mm}{$\pi_2$}}
\noLine
\UI$A \WN_i  \fCenter Y$
\UI$A \wn_i  \fCenter Y$
\BI$X \WN_i  \fCenter Y$
\DisplayProof
 & $\rightsquigarrow$ &
\bottomAlignProof
\AXC{\ \ \ $\vdots$ \raisebox{1mm}{$\pi_1$}}
\noLine
\UI$X \fCenter A$
\AXC{\ \ \,$\vdots$ \raisebox{1mm}{$\pi_2$}}
\noLine
\UI$A \WN_i  \fCenter Y$
\UI$A \fCenter Y\DWN_i $
\BI$X \fCenter Y\DWN_i $
\UI$X \WN_i \ \fCenter\ Y$
\DisplayProof
 \\
\end{tabular}
\end{center}

\begin{center}
\begin{tabular}{@{}rcl@{}}
\bottomAlignProof
\AXC{\ \ \ $\vdots$ \raisebox{1mm}{$\pi_1$}}
\noLine
\UI$\Pi \fCenter \alpha$
\UI$\Pi^\oplus \fCenter \alpha^+$
\AXC{\ \ \ $\vdots$ \raisebox{1mm}{$\pi_2$}}
\noLine
\UI$\alpha^\oplus \fCenter \Delta$
\UI$\alpha^+ \fCenter \Delta$
\BI$\Pi^\oplus \fCenter \Delta$
\DisplayProof
 & $\rightsquigarrow$ &
\bottomAlignProof
\AXC{\ \ \ $\vdots$ \raisebox{1mm}{$\pi_1$}}
\noLine
\UI$\Pi \fCenter \alpha$
\AXC{\ \ \,$\vdots$ \raisebox{1mm}{$\pi_2$}}
\noLine
\UI$\alpha^\oplus \fCenter \Delta$
\UI$\alpha \fCenter \Delta^\ominus$
\BI$\Pi \fCenter \Delta^\ominus$
\UI$\Pi^\oplus \fCenter \Delta$
\DisplayProof
 \\
\end{tabular}
\end{center}

\begin{center}
\begin{tabular}{@{}rcl@{}}
\bottomAlignProof
\AXC{\ \ \ $\vdots$ \raisebox{1mm}{$\pi_1$}}
\noLine
\UI$\Pi \fCenter \delta^\ominus$
\UI$\Pi \fCenter \delta^-$
\AXC{\ \ \ $\vdots$ \raisebox{1mm}{$\pi_2$}}
\noLine
\UI$\delta \fCenter \Delta$
\UI$\delta^- \fCenter \Delta^\ominus$
\BI$\Pi \fCenter \Delta^\ominus$
\DisplayProof
 & $\rightsquigarrow$ &
\bottomAlignProof
\AXC{\ \ \ $\vdots$ \raisebox{1mm}{$\pi_1$}}
\noLine
\UI$\Pi \fCenter \delta^\ominus$
\UI$\Pi^\oplus \fCenter \delta$
\AXC{\ \ \ $\vdots$ \raisebox{1mm}{$\pi_2$}}
\noLine
\UI$\delta \fCenter \Delta$
\BI$\Pi^\oplus \fCenter \Delta$
\UI$\Pi \fCenter \Delta^\ominus$
\DisplayProof
 \\
\end{tabular}
\end{center}

\noindent In each case above, the cut in the original derivation is strongly-uniform by assumption, and after the transformation, cuts of lower complexity are introduced which can be easily verified to be strongly-uniform for each $0\leq i\leq 1$ in the first proof and also for the remaining two proofs.

The remaining operational connectives are straightforward and left to the reader.

\newpage

\subsection{Completeness for PDL}
\label{Appendix : completeness for PDL}

\noindent
$\rule{146.2mm}{0.5pt}$

\textbf{K}\ \ \ $\alpha \mra (A \pra B) \vdash (\alpha \mra A) \pra (\alpha \mra B)$

\begin{center}
\AX$\alpha \fCenter \alpha$

\AX$\alpha \fCenter \alpha$
\AX$A \fCenter A$
\BI$\alpha \mra A \fCenter \alpha \mRA A$
\LeftLabel{\scriptsize$W \,,$}
\UI$\alpha \mra A \,, \alpha \mra (A \pra B) \fCenter \alpha \mRA A$
\UI$\alpha \mBAND (\alpha \mra A \,, \alpha \mra (A \pra B)) \fCenter A$
\AX$B \fCenter B$
\BI$A \pra B \fCenter \alpha \mBAND (\alpha \mra A \,, \alpha \mra (A \pra B)) > B$

\BI$\alpha \mra (A \pra B) \fCenter \alpha \mRA (\alpha \mBAND (\alpha \mra A \,; \alpha \mra (A \pra B)) > B)$
\UI$\alpha \mra A \,, \alpha \mra (A \pra B) \fCenter \alpha \mRA (\alpha \mBAND (\alpha \mra A \,, \alpha \mra (A \pra B)) > B)$
\UI$\alpha \mBAND (\alpha \mra A \,, \alpha \mra (A \pra B)) \fCenter \alpha \mBAND (\alpha \mra A \,, \alpha \mra (A \pra B)) > B$
\UIC{$\alpha \mBAND (\alpha \mra A \,, \alpha \mra (A \pra B)) \,, \alpha \mBAND (\alpha \mra A \,, \alpha \mra (A > B)) \fCenter B$}
\LeftLabel{\scriptsize$C \,,$}
\UIC{$\alpha \mBAND ((\alpha \mra A \,, \alpha \mra (A \pra B)) \,, (\alpha \mra A \,, \alpha \mra (A \pra B))) \fCenter B$}
\UI$(\alpha \mra A \,, \alpha \mra (A \pra B)) \,, (\alpha \mra A \,, \alpha \mra (A \pra B)) \fCenter \alpha \mRA B$
\UI$\alpha \mra A \,, \alpha \mra (A \pra B) \fCenter \alpha \mRA B$
\UI$\alpha \mra A \,, \alpha \mra (A \pra B) \fCenter \alpha \mra B$
\UI$\alpha \mra (A \pra B) \fCenter \alpha \mra A > \alpha \mra B$
\UI$\alpha \mra (A \pra B) \fCenter \alpha \mra A \pra \alpha \mra B$
\DisplayProof
\end{center}

\noindent
$\rule{146.2mm}{0.5pt}$

\textbf{Box-Choice}\ \ \ $(\alpha \gor \beta) \mra A \,\dashv\vdash\, (\alpha \mra A) \pand (\beta \mra A)$

\begin{center}
\AX$\alpha \fCenter \alpha$
\RightLabel{\scriptsize$W\between$}
\UI$\alpha \fCenter \alpha \gC \beta$
\UI$\alpha \fCenter \alpha \gor \beta$
\AX$A \fCenter A$
\BI$(\alpha \gor \beta) \mra A \fCenter \alpha \mRA A$
\UI$(\alpha \gor \beta) \mra A \fCenter \alpha \mra A$
\AX$\beta \fCenter \beta$
\RightLabel{\scriptsize$W\between$}
\UI$\beta \fCenter \alpha \gC \beta$
\UI$\beta \fCenter \alpha \gor \beta$
\AX$A \fCenter A$
\BI$(\alpha \gor \beta) \mra A \fCenter \beta \mRA A$
\UI$(\alpha \gor \beta) \mra A \fCenter \beta \mra A$
\BI$(\alpha \gor \beta) \mra A\,, (\alpha \gor \beta) \mra A \fCenter (\alpha \mra A) \pand (\beta \mra A)$
\LeftLabel{\scriptsize$C ,$}
\UI$(\alpha \gor \beta) \mra A \fCenter (\alpha \mra A) \pand (\beta \mra A)$
\DisplayProof
\\
\bigskip
\bigskip
\AX$\alpha \fCenter \alpha$
\AX$A \fCenter A$
\BI$\alpha \mra A \fCenter \alpha \mRA A$
\UI$\alpha \mBAND \alpha \mra A \fCenter A$
\UI$\alpha \fCenter A \mLA \alpha \mra A$
\AX$\beta \fCenter \beta$
\AX$A \fCenter A$
\BI$\beta \mra A \fCenter \beta \mRA A$
\UI$\beta \mBAND \beta \mra A \fCenter A$
\UI$\beta \fCenter A \mLA \beta \mra A$
\BI$\alpha \gor \beta \fCenter (A \mLA \alpha \mra A) \gC (A \mLA \beta \mra A)$
\RightLabel{\scriptsize{\emph{choice} $\vartriangleleft$}}
\UI$\alpha \gor \beta \fCenter A \mLA (\alpha \mra A \,, \beta \mra A)$
\UI$\alpha \gor \beta \mBAND (\alpha \mra A \,, \beta \mra A)\fCenter A$
\UI$\alpha \mra A \,, \beta \mra A \fCenter \alpha \gor \beta \mRA A$
\UI$(\alpha \mra A) \pand (\beta \mra A) \fCenter \alpha \gor \beta \mRA A$
\UI$(\alpha \mra A) \pand (\beta \mra A) \fCenter (\alpha \gor \beta) \mra A$
\DisplayProof
 \\
\end{center}


\noindent
$\rule{146.2mm}{0.5pt}$

\textbf{Box-Composition}\ \ \ $(\alpha \gsc \beta) \mra A \,\dashv\vdash\, \alpha \mra (\beta \mra A)$

\begin{center}
\begin{tabular}{lr}
\AX$\alpha \fCenter \alpha$
\AX$\beta \fCenter \beta$
\BI$\alpha \gSC \beta \fCenter \alpha \gsc \beta$
\AX$A \fCenter A$
\BI$(\alpha \gsc \beta) \mra A \fCenter (\alpha \gSC \beta) \mRA A$
\RightLabel{\scriptsize \emph{act}$ \vartriangleright$}
\UI$(\alpha \gsc \beta) \mra A \fCenter \alpha \mRA (\beta \mRA A)$
\UI$\alpha \mBAND (\alpha \gsc \beta) \mra A \fCenter \beta \mRA A$
\UI$\alpha \mBAND (\alpha \gsc \beta) \mra A \fCenter \beta \mra A$
\UI$(\alpha \gsc \beta) \mra A \fCenter \alpha \mRA \beta \mra A$
\UI$(\alpha \gsc \beta) \mra A \fCenter \alpha \mra (\beta \mra A)$
\DisplayProof

 & 

\AX$\alpha \fCenter \alpha$
\AX$\beta \fCenter \beta$
\AX$A \fCenter A$
\BI$\beta \mra A \fCenter \beta \mRA A$
\BI$\alpha \mra (\beta \mra A) \fCenter \alpha \mRA (\beta \mRA A)$
\RightLabel{\scriptsize{\emph{act}$ \vartriangleright$}}
\UI$\alpha \mra (\beta \mra A) \fCenter (\alpha \gSC \beta) \mRA A$
\UI$(\alpha \gSC \beta) \mBAND \alpha \mra (\beta \mra A) \fCenter A$
\UI$\alpha \gSC \beta \fCenter A \mLA \alpha \mra (\beta \mra A)$
\UI$\alpha \gsc \beta \fCenter A \mLA \alpha \mra (\beta \mra A)$
\UI$\alpha \gsc \beta \mBAND \alpha \mra (\beta \mra A) \fCenter A$
\UI$\alpha \mra (\beta \mra A) \fCenter \alpha \gsc \beta \mRA A$
\UI$\alpha \mra (\beta \mra A) \fCenter (\alpha \gsc \beta) \mra A$
\DisplayProof
 \\
\end{tabular}
\end{center}


\noindent
$\rule{146.2mm}{0.5pt}$

\textbf{Box-Test}\ \ \ $A\wn \mra B \,\dashv\vdash\, A \pra B$

\begin{center}
\begin{tabular}{lr}
\AX$A \fCenter A$
\AX$B \fCenter B$
\BI$A \pra B \fCenter A > B$
\RightLabel{\scriptsize{$\wn \vartriangleright$}}
\UI$A \pra B \fCenter A\WN \mRA B$
\UI$A\WN \mBAND A \pra B \fCenter B$
\UI$A\WN \fCenter B \mLA A \pra B$
\UI$A\wn \fCenter B \mLA A \pra B$
\UI$A\wn \mBAND A \pra B \fCenter B$
\UI$A \pra B \fCenter A\wn \mRA B$
\UI$A \pra B \fCenter A\wn \mra B$
\DisplayProof

& 

\AX$A \fCenter A$
\UI$A\WN \fCenter A\wn$
\AX$B \fCenter B$
\BI$A \wn \mra B \fCenter A \WN \mRA B$
\RightLabel{\scriptsize{$\wn \vartriangleright$}}
\UI$A \wn \mra B \fCenter A > B$
\UI$A \wn \mra B \fCenter A \pra B$
\DisplayProof
 \\
\end{tabular}
\end{center}


\noindent
$\rule{146.2mm}{0.5pt}$

\textbf{Box-Distributivity}\ \ \ $\alpha \mra (A \pand B) \,\dashv\vdash\, \alpha \mra A \pand \alpha \mra B$

\begin{center}
\AX$\alpha \fCenter \alpha$
\AX$A \fCenter A$
\LeftLabel{\scriptsize$W,$}
\UI$A\,, B \fCenter A$
\UI$A \pand B \fCenter A$
\BI$\alpha \mra (A \pand B) \fCenter \alpha \mRA A$
\UI$\alpha \mra (A \pand B) \fCenter \alpha \mra A$
\AX$\alpha \fCenter \alpha$
\AX$B \fCenter B$
\LeftLabel{\scriptsize$W,$}
\UI$A\,, B \fCenter B$
\UI$A \pand B \fCenter B$
\BI$\alpha \mra (A \pand B) \fCenter \alpha \mRA B$
\UI$\alpha \mra (A \pand B) \fCenter \alpha \mra B$
\BI$\alpha \mra (A \pand B)\,, \alpha \mra (A \pand B) \fCenter (\alpha \mra A) \pand (\alpha \mra B)$
\LeftLabel{\scriptsize$C ,$}
\UI$\alpha \mra (A \pand B) \fCenter (\alpha \mra A) \pand (\alpha \mra B)$
\DisplayProof
\\
\bigskip
\bigskip
\AX$\alpha \fCenter \alpha$
\AX$A \fCenter A$
\BI$\alpha \mra A \fCenter \alpha \mRA A$
\UI$\alpha \mBAND \alpha \mra A \fCenter A$

\AX$\alpha \fCenter \alpha$
\AX$B \fCenter B$
\BI$\alpha \mra B \fCenter \alpha \mRA B$
\UI$\alpha \mBAND \alpha \mra B \fCenter B$

\BI$(\alpha \mBAND \alpha \mra A)\,, (\alpha \mBAND \alpha \mra B) \fCenter A \pand B$
\LeftLabel{\scriptsize$mon \mband$}
\UI$\alpha \mBAND (\alpha \mra A\,, \alpha \mra B) \fCenter A \pand B$
\UI$\alpha \mra A\,, \alpha \mra B \fCenter \alpha \mRA A \pand B$
\UI$\alpha \mra A \pand \alpha \mra B \fCenter \alpha \mRA A \pand B$
\UI$\alpha \mra A \pand \alpha \mra B \fCenter \alpha \mra (A \pand B)$
\DisplayProof
 \\
\end{center}

\bibliography{BIB}
 \bibliographystyle{plain}

\end{document}